\documentclass[10pt, letterpaper]{amsart}

\usepackage{bm}
\usepackage{latexsym, amsmath, amstext, amssymb, amsfonts, amscd, bm, array, multirow, amsbsy, mathrsfs}
\usepackage{amsthm}
\usepackage{t1enc}
\usepackage[mathscr]{eucal}
\usepackage{indentfirst}
\usepackage{pb-diagram}
\usepackage{graphicx}
\usepackage{fancyhdr}
\usepackage{fancybox}
\usepackage{enumerate}
\usepackage{color}
\usepackage{tikz-cd}
\usepackage[all]{xy}
\usepackage{hyperref}
\usepackage{tikz}
\usepackage{xparse}
\hypersetup{colorlinks=false,pdfborderstyle={/S/U/W 0}}
\usetikzlibrary{matrix}

\usepackage{url}
\usepackage[sort&compress,comma]{natbib}
\bibpunct{[}{]}{,}{n}{}{,}
\theoremstyle{plain}
\newtheorem{thm}{Theorem}[section]

\newtheorem{prop}[thm]{Proposition}

\newtheorem{lemma}[thm]{Lemma}
\newtheorem{cor}[thm]{Corollary}

\theoremstyle{definition}
\newtheorem{definition}[thm]{Definition}

\theoremstyle{remark}
\newtheorem{remark}[thm]{Remark}
\newtheorem{example}[thm]{Example}

\newtheorem*{ack}{Acknowledgements}

\newcommand{\im}{\mathrm{im}}

\newcommand{\End}{\mathrm{End}}
\newcommand{\Aut}{\mathrm{Aut}}

\newcommand{\Rep}{{\rm Rep}}
\newcommand{\Irrep}{{\rm Irrep}}

\newcommand{\eqdef}{\stackrel{{\rm def.}}{=}}

\DeclareFontFamily{U}{rsf}{}
\DeclareFontShape{U}{rsf}{m}{n}{<5> <6> rsfs5 <7> <8> <9> rsfs7 <10-> rsfs10}{}
\DeclareMathAlphabet\Scr{U}{rsf}{m}{n}

\def\Z{\mathbb{Z}}
\def\C{\mathbb{C}}
\def\R{\mathbb{R}}

\def\H{\mathbb{H}}
\def\K{\mathbb{K}}

\def\rk{{\rm rk}}

\def\Ad{\mathrm{Ad}}

\def\rGamma{\mathrm{\Gamma}}

\def\Diff{\mathrm{Diff}}

\newcommand{\be}{\begin{equation*}}
\newcommand{\ee}{\end{equation*}}
\newcommand{\ben}{\begin{equation}}
\newcommand{\een}{\end{equation}}
\newcommand{\beqa}{\begin{eqnarray*}}
	\newcommand{\eeqa}{\end{eqnarray*}}
\newcommand{\beqan}{\begin{eqnarray}}
\newcommand{\eeqan}{\end{eqnarray}}

\newcommand{\id}{\mathrm{id}}

\newcommand{\lhookrightarrow}{\ensuremath{\lhook\joinrel\relbar\joinrel\rightarrow}}

\def\cR{{\mathcal R}}

\def\cC{{\mathcal C}}

\def\Cl{\mathrm{Cl}}

\def\cK{\mathcal{K}}

\def\Spin{\mathrm{Spin}}

\def\Pin{\mathrm{Pin}}
\def\Spin{\mathrm{Spin}}

\def\i{\mathbf{i}}

\def\SO{\mathrm{SO}}
\def\O{\mathrm{O}}
\def\U{\mathrm{U}}

\def\cI{\mathcal{I}}

\def\cN{\mathcal{N}}

\def\cC{\mathcal{C}}

\def\G_2{\mathrm{G_2}}

\newcommand{\Hom}{{\rm Hom}}

\newcommand{\Iso}{{\rm Iso}}
\def\Quad{\mathrm{Quad}}

\def\Aut{\mathrm{Aut}}
\def\Ob{\mathrm{Ob}}
\def\ClB{\mathrm{ClB}}
\def\mClB{\mathbb{C}\mathrm{lB}}
\def\Alg{\mathrm{Alg}}

\def\mClRep{\mathrm{\mathbb{C}lRep}}

\def\tAd{\widetilde{\Ad}}

\def\ttau{\tilde{\tau}}

\def\G{\mathrm{G}}
\def\L{\mathrm{L}}

\def\R{\mathbb{R}}

\def\rS{\mathrm{S}}
\def\w{\mathrm{w}}

\def\mCl{\mathbb{C}\mathrm{l}}
\def\rGamma{\mathrm{\Gamma}}

\begin{document}

\title{Complex Lipschitz structures and bundles of complex Clifford modules}

\author[C. Lazaroiu]{C. Lazaroiu} \address{Center for Geometry and
  Physics, Institute for Basic Science, Pohang, Republic of Korea}
\email{calin@ibs.re.kr}

\author[C. S. Shahbazi]{C. S. Shahbazi} \address{Department of Mathematics, University of Hamburg, Germany}
\email{carlos.shahbazi@uni-hamburg.de}


\begin{abstract}
Let $(M,g)$ be a pseudo-Riemannian manifold of signature $(p,q)$. We
construct mutually quasi-inverse equivalences between the groupoid of
bundles of weakly-faithful complex Clifford modules on $(M,g)$ and the
groupoid of reduced complex Lipschitz structures on $(M,g)$. As an
application, we show that $(M,g)$ admits a bundle of irreducible
complex Clifford modules if and {\em only if} it admits either a
$\Spin^{c}(p,q)$ structure (when $p+q$ is odd) or a $\Pin^{c}(p,q)$
structure (when $p+q$ is even). When $p-q\equiv_8 3,4,6, 7$, we compare with the
classification of bundles of irreducible real Clifford modules which
we obtained in previous work. The results obtained in this note form a
counterpart of the classification of bundles of faithful complex
Clifford modules which was previously given by T. Friedrich and
A. Trautman.
\end{abstract}

\maketitle

\setcounter{tocdepth}{1} 
\tableofcontents


\section{Introduction}


In reference \cite{FriedrichTrautman}, T.~Friedrich and A.~Trautman
classified bundles of {\em faithful} complex Clifford modules over a
connected pseudo-Riemannian manifold $(M,g)$ in terms of complex
Lipschitz structures (see also \cite{BobienskiTrautman}). In this
note, we establish a general equivalence of categories between the
groupoid of {\em weakly-faithful}\footnote{See Definition
  \ref{def:wf}.} complex Clifford modules over a connected
pseudo-Riemannian manifold of arbitrary signature $(p,q)$ and the
groupoid of complex Lipschitz structures associated to the
corresponding fiberwise representation.  This equivalence, which we
give in Proposition \ref{propdef:functors} and Theorem
\ref{thm:BundleLipschitz}, clarifies and extends the correspondence
given in \cite[Sec. 5, Theorems 2 and 3]{FriedrichTrautman} between
bundles of faithful complex Clifford modules and the corresponding
complex Lipschitz structures.

The fibers of a bundle $S$ of {\em irreducible} complex Clifford
modules are faithful iff the dimension $d=p+q$ of $M$ is even (in
which case they are ``Dirac representations'' in the language of
\cite{FriedrichTrautman}). In this case, the results of
\cite{FriedrichTrautman} imply that the corresponding complex
Lipschitz structure is homotopy-equivalent with a $\Pin^c(p,q)$
structure.  When $d$ is odd, the fibers of a bundle $S$ of irreducible
complex Clifford modules are weakly-faithful but not faithful (they
are ``Pauli representations'' in the language of
\cite{FriedrichTrautman, BobienskiTrautman}). In this case, we show
that the corresponding complex Lipschitz structure is
homotopy-equivalent with a $\Spin^c(p,q)$ structure. In particular,
this shows that an odd-dimensional pseudo-Riemannian manifold admits a
bundle of irreducible complex Clifford modules if and {\em only if} it
admits a $\Spin^c(p,q)$ structure, in which case any such bundle is
associated to the appropriate $\Spin^c(p,q)$ structure through the
tautological representation. We also show that the classifications of
$\Spin^c(p,q)$ structures and of bundles of irreducible complex
Clifford modules agree and hence the isomorphism classes of the latter
form a torsor over $H^2(M,\mathbb{Z})$ when $(M,g)$ admits a
$\Spin^c(p,q)$ structure.

Reference \cite{Lipschitz} solved a similar classification problem for
bundles of irreducible {\em real} Clifford modules in terms of real
Lipschitz structures. The two problems are related by the faithful
{\em realification functor}, which maps a bundle $S$ of complex
Clifford modules to the underlying bundle of real Clifford
modules. When the fibers of $S$ are complex-irreducible, they are also
real irreducible exactly when $p-q\equiv_8 3,4,6,7$, so in these cases
it is natural to compare the two categories. For such signatures, we
show that the realification functor is faithful and strictly
surjective but not full and we describe its preimage. When $M$ is
unorientable with $p-q\equiv_8 3,7$, the pseudo-Riemannian manifold
$(M,g)$ does not admit bundles of complex irreducible Clifford
modules. However, such $(M,g)$ can admit $\Spin^o(p,q)$ structures
(and hence bundles of real irreducible Clifford modules) provided that
the topological conditions given in \cite{Lipschitz} are satisfied.

\

\paragraph{\bf Notations and conventions}

All manifolds considered are smooth, paracompact, Hausdorff and
connected.  All fiber bundles and sections thereof are smooth. For a
field $\K\in \{\R,\C\}$, let $\Alg_\K$ denote the category of unital
associative $\K$-algebras and unital algebra morphisms. The unit
groupoid of any category $\cC$ is denoted by $\cC^\times$; it is
defined as the groupoid obtained from $\cC$ by keeping only the
invertible morphisms.


\section{Complexified Clifford algebras and complex Clifford groups}


A {\em real quadratic space} is a pair $(V,h)$, where $V$ is a
finite-dimensional $\R$-vector space and $h$ is a symmetric and
non-degenerate $\R$-bilinear pairing defined on $V$.  Given a real
quadratic space $(V,h)$, let $\Cl(V,h)$ denote its real Clifford
algebra. By definition, the {\em complexified Clifford algebra}
of $(V,h)$ is the complexification $\mCl(V,h) \eqdef
\Cl(V,h)\otimes_{\R}\C$. There exists a natural unital isomorphism of
$\C$-algebras:
\be
\mCl(V,h) \simeq_{\Alg_\C} \Cl(V_\C,h_\C)~~, 
\ee
where $V_\C\eqdef V\otimes_\R \C$ is the complexification of $V$ and
$h_\C:V_\C\times V_\C\rightarrow \C$ is the $\C$-bilinear extension of
$h:V\times V\rightarrow \R$. 

\subsection{The categories $\Cl_\C$ and $\Cl_\R$}

Let $\Quad$ be the category whose objects are real quadratic spaces
and whose arrows are isometries between such and let $\Quad^{\times}$
denote the unit groupoid of $ \Quad$. The Clifford algebra
construction gives two functors:
\begin{itemize}
\itemsep 0.0em
\item The {\em real Clifford functor} $\Cl\colon\Quad\rightarrow
  \Alg_\R$, which maps a real quadratic space $(V,h)$ to its real
  Clifford algebra and maps an isometry $\varphi_{0}\colon (V,h)\rightarrow
  (V^{\prime},h^{\prime})$ to the map $\Cl(\varphi_{0})\colon
  \Cl(V,h)\rightarrow \Cl(V^{\prime},h^{\prime})$ defined as the
  unique unital morphism of associative $\R$-algebras which satisfies
  the condition $\Cl(\varphi_{0})|_V = \varphi_{0}$. The image of this
  functor is a {\em non-full} sub-category of $\Alg_\R$ which we
  denote by $Cl_\R$.
\item The {\em complexified Clifford functor}
  $\mCl\colon\Quad\rightarrow \Alg_{\C}$, which maps a quadratic real
  vector space $(V,h)$ to its complexified Clifford algebra
  $\mCl(V,h)$ and maps an isometry $\varphi_{0}\colon (V,h)\rightarrow
  (V^{\prime},h^{\prime})$ to the $\C$-linear extension
  $\mCl(\varphi_{0})\eqdef \Cl(\varphi_{0})\otimes_\R \C$ of
  $\Cl(\varphi_0)$. The image of this functor is a {\em non-full}
  sub-category of $\Alg_{\C}$ which we denote by $Cl_{\mathbb{C}}$.
\end{itemize}

The corestrictions of the functors $\Cl$ and $\mCl$ to their images
give isomorphisms (i.e., strictly surjective equivalences) of
categories between $\Quad$ and $Cl_\R$, respectively $\Quad$ and
$Cl_{\C}$. The category $\Quad$ admits a skeleton whose objects are
the {\em standard quadratic spaces} $\R^{p,q}\eqdef
(\R^{p+q},h_{p,q})$, where $h_{p,q}:\R^{p+q}\times \R^{p+q}\rightarrow
\R$ is the standard symmetric bilinear form of signature $(p,q)$. The
objects of this skeleton form a countable set indexed by the pairs
$(p,q)\in \Z_{\geq 0} \times \Z_{\geq 0}$. Accordingly, the category
$Cl_\R$ admits a skeleton whose objects are given by the real Clifford
algebras $\Cl_{p,q}\eqdef (\R^{p,q},h_{p,q})$, while $Cl_\C$ admits a
skeleton whose objects are given by the complex Clifford algebras
$\mCl_{p,q}\eqdef \mCl(\R^{p,q},h_{p,q})=\Cl_{p,q}\otimes_\R \C$.

\begin{remark}
A morphism $\psi\colon \mCl(V,h)\rightarrow
\mCl(V^{\prime},h^{\prime})$ in the category $Cl_{\mathbb{C}}$ is a
morphism of unital associative algebras which satisfies
$\psi(V)\subset V^{\prime}$ and hence is necessarily of the form
$\psi=\mCl(\varphi_{0})$ for a uniquely-determined isometry
$\varphi_{0} \colon (V,h)\rightarrow (V^{\prime},h^{\prime})$, given
by $\varphi_{0}\eqdef \psi|_V$. Notice that $Cl_{\C}$ does not contain
any isomorphism between complexified Clifford algebras associated to
quadratic real vector spaces of different signatures.
\end{remark}

\subsection{The complex Clifford group}

Let $(V,h)$ be a quadratic real vector space. Let $\pi$ be the parity
involution of $\Cl(V,h)$, $\tau$ be the reversion anti-automorphism
and $\ttau\eqdef \tau\circ \pi$ be the twisted reversion. Let
$\pi_{\C}$ be automorphism of $\mCl(V,h)$ given by the $\C$-linear
extension of $\pi$ and $\tau_{\C}$ be the {\em antilinear} extension
of $\tau$:
\begin{equation*}
\tau_{\C}(x\otimes_\R z) = \tau(x)\otimes_\R \bar{z}\, , \quad x\in \Cl(V,h)\, , \quad z\in \mathbb{C}\, .
\end{equation*}
Finally, let $\ttau_\C\eqdef \tau_{\C}\circ\pi_{\C}$.

Let $\mCl(V,h)^\times$ be the group of invertible elements of the
algebra $\mCl(V,h)$ and $\tAd:\mCl(V,h)^\times\rightarrow
\Aut_\C(\mCl(V,h))$ be the twisted adjoint action:
\begin{equation*}
\tAd(x)(y) = \pi(x) y x^{-1}~~~~\forall~ x\in \mCl(V,h)^\times~~\forall~y\in V_\C~~.
\end{equation*}
Recall that the {\em twisted norm} of $\mCl(V,h)$ is the map $N\colon
\mCl(V,h) \to \mCl(V,h)$ given by:
\begin{equation*}
N(x) \eqdef \tilde{\tau}_{\C}(x) x\, , \qquad\forall\, x\in \mCl(V,h)\, .
\end{equation*}

\noindent Also recall the following (see \cite{ABS}):

\begin{definition}
The {\em complex Clifford group} is the subgroup of $\mCl(V,h)^\times$
defined through:
\begin{equation*}
\rGamma(V,h) \eqdef \left\{ x\in \mCl(V,h)^\times \, | \, \tAd(x)(V)=V \right\}\, .
\end{equation*}
The {\em complex special Clifford group} is the subgroup consisting of
all even elements of $\rGamma(V,h)$:
\begin{equation*}
\rGamma^{s}(V,h) \eqdef \rGamma(V,h)\cap \mCl^{+}(V,h)~~,
\end{equation*}
where $\mCl^{+}(V,h)$ denotes the ``even subalgebra'' of $\mCl(V,h)$. 
\end{definition}

\

\noindent The twisted norm restricts to a group morphism $N\colon \Gamma(V,h) \to
\mathbb{C}^\times$.

\begin{remark}
We have $\rGamma(V,h) \subset \G(V,h)$, where: 
\be
\G(V,h)\eqdef \left\{x\in \mCl(V,h)^\times \quad | \quad \tAd(x)(V_\C)=V_\C \right\}
\ee
is the ordinary Clifford group of $(V_\C,h_\C)$. Moreover, an element $x$ of
$\G(V,h)$ belongs to $\rGamma(V,h)$ if and only if the orthogonal
transformation $\tAd(x)\in \O(V_{\C},h_{\C})$ preserves the real
subspace $V\subset V_\C$. Equivalently, if $c\colon V_{\C}\to V_{\C}$
is the real structure (antilinear involution) on $V_{\C}$ with real
subspace (subspace of fixed points) $V\subset V_{\C}$, then an element
$x$ of $\G(V,h)$ belongs to $\Gamma(V,h)$ if and only if $c\circ
\tAd(x) = \tAd(x)\circ c$. In contrast with $\G(V,h)$, the complex
Clifford group $\Gamma(V,h)$ depends on the signature of the real quadratic
space $(V,h)$.
\end{remark}

\

\noindent Finally, recall that the groups $\Pin^{c}(V,h)$ and $\Spin^{c}(V,h)$ are defined as follows \cite{ABS,Karoubi}:
\begin{eqnarray*}
\Pin^{c}(V,h) &=& \mathrm{Ker}(|N|\colon \Gamma(V,h)\to \mathbb{R}_{>0})~~,\\
\Spin^{c}(V,h) &=& \mathrm{Ker}(|N|\colon \Gamma^{s}(V,h)\to \mathbb{R}_{>0})~~.
\end{eqnarray*}

\begin{prop}
\label{prop:Cliffordgrouphomotopy}
One has short exact sequences:
\begin{eqnarray*}
1\longrightarrow \Pin^{c}(V,h) \lhookrightarrow \Gamma(V,h) &\stackrel{|N|}{\longrightarrow}& \mathbb{R}_{>0}\longrightarrow 1~~,\\
1\longrightarrow \Spin^{c}(V,h)\lhookrightarrow \Gamma^{s}(V,h) &\stackrel{|N|}{\longrightarrow}& \mathbb{R}_{>0}\longrightarrow 1~~.
\end{eqnarray*}
Moreover $\rGamma(V,h)$ is homotopy-equivalent with $\Pin^{c}(V,h)$
and $\rGamma^{s}(V,h)$ is homotopy-equivalent with $\Spin^{c}(V,h)$.
\end{prop}

\begin{proof}
Exactness of both sequences is obvious. The map:
\beqan
r\colon \rGamma(V,h) &\rightarrow& \Pin^{c}(V,h)\, , \nonumber\\
r(x) &\eqdef& \frac{x}{\sqrt{|N(x)|}}\, , \label{r}
\eeqan
is a homotopy retraction of $\rGamma(V,h)$ onto $\Pin^{c}(V,h)$. The
proof for $\Spin^c(V,h)$ is similar, the corresponding homotopy retraction 
being given by $r|_{\rGamma^s(V,h)}$. Notice the relation:
\be
\Ad(r(x))=\Ad(x)~~\forall x\in \rGamma(V,h)~~.
\ee
\end{proof}

\noindent
We have $\Gamma(V,h)\simeq \Pin^c(V,h)\cdot \C^\times$ and
$\Gamma^s(V,h)\simeq \Spin^c(V,h)\cdot \C^\times$.  The following
theorem summarizes the key properties of $\Gamma(V,h)$,
$\Gamma^{s}(V,h)$, $\Pin^{c}(V,h)$ and $\Spin^{c}(V,h)$.

\begin{thm}\cite{ABS}
There exist short exact sequences:
\begin{eqnarray*}
1\longrightarrow \C^\times \lhookrightarrow \Gamma(V,h) &\stackrel{\tAd}{\longrightarrow}& \O(V,h)\longrightarrow 1 \, ,\\
1\longrightarrow \U(1)\lhookrightarrow \Pin^{c}(V,h) &\stackrel{\tAd}{\longrightarrow}& \O(V,h)\longrightarrow 1 \, ,
\end{eqnarray*}
where $\U(1)$ is the subgroup of $\mCl(V,h)^\times$ consisting of
elements of the form $1\otimes z$ with $|z|=1$. Likewise, one has
short exact sequences:
\begin{eqnarray*}
1\longrightarrow \C^\times \lhookrightarrow \Gamma^{s}(V,h) &\stackrel{\tAd}{\longrightarrow}& \SO(V,h)\longrightarrow 1 \, ,\\
1\longrightarrow \U(1)\lhookrightarrow \Spin^{c}(V,h) &\stackrel{\tAd}{\longrightarrow}& \SO(V,h)\longrightarrow 1~~.
\end{eqnarray*}
Moreover, there exist canonical isomorphisms of groups:
\begin{equation*}
\Pin^{c}(V,h)\simeq \Pin(V,h)\cdot \U(1)\eqdef \left(\Pin(V,h)\times \U(1)\right)/\left\{ (1,1),(-1,-1)\right\}
\end{equation*} 
and:
\begin{equation*}
\Spin^{c}(V,h)\simeq \Spin(V,h)\cdot \U(1)\eqdef \left(\Spin(V,h)\times \U(1)\right)/\left\{ (1,1),(-1,-1)\right\}~~.
\end{equation*} 
\end{thm}


\section{Complex Clifford representations and complex Lipschitz groups}


\noindent Let $(V,h)$ be a real quadratic space.

\begin{definition}
A {\em complex Clifford representation} is a unital morphism of
associative $\R$-algebras $\gamma\colon
\Cl(V,h)\rightarrow \End_{\C}(S)$, where $S$ is a finite-dimensional
vector space over $\C$. 
\end{definition}

A complex Clifford representation $\gamma$ endows $S$ with the
structure of (unital) left module over $\Cl(V,h)$, while the
multiplication of vectors with complex scalars endows $S$ with a
compatible structure of (unital) finite rank (and necessarily free)
right module over the field of complex numbers. Conversely, any such
bimodule (which we shall call a {\em complex Clifford module}) can be
viewed as a complex Clifford representation.

Extending $\gamma$ by complex linearity gives a unital morphism of
associative $\C$-algebras:
\begin{equation}
\label{eq:Cextension}
\gamma_\C\colon \mCl(V,h)\rightarrow \End_{\C}(S)~.
\end{equation}
Conversely, every unital morphism of associative $\C$-algebras from
$\mCl(V,h)$ to $\End_\C(S)$ restricts to a complex representation of
$\Cl(V,h)$. Any complex Clifford representation
$\gamma\colon\Cl(V,h)\rightarrow \End_\C(S)$ also induces a real
Clifford representation $\gamma_{\R}\colon
\Cl(V,h)\rightarrow \End_\R(S_{\mathbb{R}})$ on the
\emph{realification} $S_{\mathbb{R}}$ of $S$ (defined as the
underlying real vector space of $S$), which is an $\R$-vector space of
dimension $\dim_{\mathbb{R}} S_{\mathbb{R}} = 2 \dim_{\mathbb{C}} S$.

\subsection{Unbased morphisms of complex Clifford representations}

Let $\gamma\colon \Cl(V,h)\rightarrow \End_\C(S)$ and $\gamma^{\prime}
\colon\Cl(V^{\prime},h^{\prime})\rightarrow \End_\C(S^{\prime})$ be
two complex Clifford representations.

\begin{definition}
\label{DefClMf}
A {\em morphism of complex Clifford representations} from $\gamma$ to
$\gamma^{\prime}$ is a pair $(\varphi_0,\varphi)$ such that:
\begin{enumerate}[1.]
\itemsep 0.0em
\item $\varphi_0 \colon V\rightarrow V^{\prime}$ is an isometry from
  $(V,h)$ to $(V^{\prime},h^{\prime})$
\item $\varphi\colon S\rightarrow S^{\prime}$ is a $\C$-linear map
\item $\gamma^{\prime}(\Cl(\varphi_0)(x))\circ \varphi = \varphi \circ
  \gamma(x)$ for all $x\in\Cl(V,h)$.
\end{enumerate}
A morphism of complex Clifford representations is called {\em
  based} if $V'=V$ and $\varphi_0=\id_V$. A (not necessarily based)
isomorphism of complex Clifford representations from $\gamma$ to
itself is called an {\em automorphism} of $\gamma$.
\end{definition}

In our language, a morphism of complex representations in the
traditional sense corresponds to a \emph{based} morphism. Since
$\Cl(V,h)$ is generated by $V$ and $\Cl(V^{\prime},h^{\prime})$ is
generated by $V^{\prime}$ while the morphism $\Cl(\varphi_{0})$ is
$\R$-linear, condition 3. in Definition \ref{DefClMf} is equivalent
with:
\begin{equation*}
\gamma^{\prime} (\varphi_0(v))\circ \varphi = \varphi \circ \gamma(v)~~~\forall\, v\in V\, ,
\end{equation*}
which can also be written as: 
\begin{equation*}
R_{\varphi}\circ \gamma^{\prime} \circ \varphi_0 = L_{\varphi} \circ \gamma|_V\, ,
\end{equation*}

\noindent
or
\begin{equation*}
R_{\varphi}\circ \gamma^{\prime}\circ \Cl(\varphi_0)=L_{\varphi} \circ \gamma\, ,
\end{equation*}

\noindent
where $L_{\varphi}\colon\End_\C(S)\rightarrow \Hom_\C(S,S')$ and
$R_{\varphi}\colon\End_\C(S')\rightarrow \Hom_\C(S,S')$ are defined as follows:
\begin{equation*}
L_{\varphi}(A)\eqdef \varphi \circ A\, ,\quad R_{\varphi}(B)\eqdef B\circ \varphi\, , \quad \forall\,\, A\in \End_\C(S)\, , \quad\forall\,\, B\in \End_\C(S')\, .
\end{equation*}

\begin{equation*}
\xymatrix{
\mathrm{Cl}(V^{\prime},h^{\prime})~  \ar[rr]^{\gamma^{\prime}}
&& ~\End_\C(S^{\prime}) ~ \ar[d]^{R_{\varphi}} \\
&& \Hom_\C(S,S')\\
\mathrm{Cl}(V,h) \ar[uu]^{\mathrm{Cl}(\varphi_0)} \ar[rr]_{\gamma} && \End_\C(S) \ar[u]^{L_{\varphi}} 
}
\end{equation*}

\noindent	
With morphisms defined as above, complex Clifford representations form
a category denoted $\mClRep$, were compatible morphisms
$(\varphi_0,\varphi)$ and $(\varphi^{\prime}_0,\varphi^{\prime})$
compose pairwise, that is  $(\varphi_0^{\prime},\varphi^{\prime})\circ
(\varphi_0,\varphi)\eqdef (\varphi_0^{\prime} \circ
\varphi_0,\varphi^{\prime} \circ \varphi)$. The functor
$\Pi:\mClRep\rightarrow Cl_\R$ which takes $\gamma$ into $\Cl(V,h)$
and $(\varphi_0,\varphi)$ into $\Cl(\varphi_{0})$ is a fibration whose
fiber above $\Cl(V,h)$ is the usual category $\Rep_\C(\Cl(V,h))$ of
complex representations of $\Cl(V,h)$ (whose morphisms are the based
morphisms of complex representations). Isomorphisms in
$\Rep_\C(\Cl(V,h))$ are the usual equivalences of complex
representations. Hence equivalences of complex representations of real
Clifford algebras coincide with based isomorphisms of $\mClRep$; in
particular, any isomorphism class of complex Clifford representations
in the category $\mClRep$ decomposes as a disjoint union of
equivalence classes. A morphism $(\varphi_0,\varphi)$ is an
isomorphism in $\mClRep$ if and only if both $\varphi_0$ and $\varphi$
are bijective.
	
\begin{prop}
\label{prop:Ad}
Let $(\varphi_0,\varphi)\colon\gamma\stackrel{\sim}{\rightarrow}
\gamma^{\prime}$ be an isomorphism of complex Clifford
representations. Then $(\varphi_{0},\varphi)$ satisfies
$\Ad(\varphi)(\gamma(V))=\gamma^{\prime}(V^{\prime})$ and
$\gamma^{\prime}\circ \varphi_0=\Ad(\varphi)\circ \gamma|_V$, where
the unital isomorphism of $\C$-algebras
$\Ad(\varphi)\colon\End_\C(S)\rightarrow \End_\C(S^{\prime})$ is
defined through:
\begin{equation*}
\Ad(\varphi)(A)\eqdef \varphi\circ A \circ \varphi^{-1}
\end{equation*}
for all $A\in \End_\C(S)$. 
\end{prop}

\begin{proof}
When $(\varphi_0,\varphi)$ is an isomorphism, condition 3. in
Definition \ref{DefClMf} becomes:
\begin{equation*}
\gamma^{\prime}\circ \Cl(\varphi_0) =\Ad(\varphi)\circ \gamma~,
\end{equation*}
being equivalent with the condition $\gamma^{\prime}\circ
\varphi_0=\Ad(\varphi)\circ \gamma|_V$, which states that $\varphi$
       {\em implements} the isometry $\varphi_0\colon (V,h)\rightarrow
       (V^{\prime},h^{\prime})$ at the level of the representation
       spaces. The Proposition now follows from the fact that
       $\Cl(\varphi_0)|_V=\varphi_0$ and $\varphi_{0}(V) =
       V^{\prime}$.
\end{proof}	

\subsection{Weakly-faithful complex Clifford representations}

\begin{definition}
\label{def:wf}
A complex Clifford representation $\gamma\colon
\Cl(V,h)\rightarrow \End_\C(S)$ is called {\em weakly-faithful} if the
restriction $\gamma_0\eqdef \gamma|_V:V\rightarrow \End_\C(S)$ is an
injective map.
\end{definition}

\begin{remark}
Notice that $\gamma|_V=\gamma_\C|_V$ and that any faithful complex
Clifford representation is weakly-faithful.
\end{remark}

\noindent
Let $\mClRep_w$ denote the full sub-category of $\mClRep$ whose
objects are the weakly-faithful complex Clifford representations and
let $\mClRep_w^\times$ denote the corresponding unit groupoid. When
$\gamma$ and $\gamma^{\prime}$ are weakly-faithful and
$(\varphi_0,\varphi)\colon \gamma\rightarrow \gamma^{\prime}$ is an
{\em isomorphism} of Clifford representations, Proposition
\ref{prop:Ad} shows that $\varphi_0$ is uniquely determined by
$\varphi$ through the relation:
\begin{equation}
\label{AdRelation}
\varphi_0 = (\gamma^{\prime}|_{V^{\prime}})^{-1}\circ \Ad(\varphi)\circ \gamma|_V\, .
\end{equation}

\noindent
It is easy to see that the converse also holds, so we have:
\begin{prop}
\label{prop:AdRelation}
Assume that $\gamma$ and $\gamma^{\prime}$ are weakly-faithful. Then
any isomorphism
$(\varphi_0,\varphi)\colon\gamma\stackrel{\sim}{\rightarrow}
\gamma^{\prime}$ is determined by the linear isomorphism
$\varphi\colon S\stackrel{\sim}{\rightarrow} S^{\prime}$. We have
$\Ad(\varphi)(\gamma(V))=\gamma^{\prime}(V^{\prime})$ and $\varphi_0$
is given by relation \eqref{AdRelation}. Conversely, any linear
isomorphism $\varphi\colon S\rightarrow S'$ which satisfies
$\Ad(\varphi)(\gamma(V))=\gamma^{\prime}(V^{\prime})$ determines an
isomorphism of quadratic spaces $\varphi_0:(V,h)\rightarrow
(V^{\prime},h^{\prime})$ through the relation \eqref{AdRelation} and
we have $(\varphi_0,\varphi)\in
\Hom_{\mClRep^\times_{w}}(\gamma,\gamma^{\prime})$.
\end{prop}

\noindent
In view of this, we denote isomorphisms of weakly-faithful complex
Clifford representations only by $\varphi$ (since $\varphi$ determines
$\varphi_0$ in this case). From the previous proposition, we obtain:

\begin{cor}
\label{cor:Homequiv}
The group $\Hom_{\mClRep_w^\times}(\gamma,\gamma^{\prime})$ can be
identified with the following subset of the set 
$\Iso_\C(S,S')$ of linear isomorphisms from $S$ to $S'$:
\begin{eqnarray*}
\label{isomwf}
\Hom_{\mClRep_w^\times}(\gamma,\gamma^{\prime}) \equiv\{\varphi\in \Iso_\C(S,S')|\Ad(\varphi)(\gamma(V))=\gamma^{\prime}(V^{\prime})\}\, ,\quad \gamma,\gamma^{\prime}\in \Ob (\mClRep_w)~.
\end{eqnarray*}
\end{cor}

\subsection{Complex Lipschitz groups}

When $\gamma$ is weakly-faithful, we can identify $V$ with its image
$W\eqdef \gamma(V)$ inside $\End_{\C}(S)$. Equip $W$ with the bilinear
form $\mu$ induced by $\gamma$ from $(V,h)$, so that $(W,\mu)$ is a
real quadratic space isometric with $(V,h)$.

\begin{definition}
The {\em complex Lipschitz group} of a weakly-faithful complex
Clifford representation $\gamma:\Cl(V,h)\rightarrow \End_\C(S)$ is
defined as the following sub-group of $\Aut_\C(S)$:
\begin{equation*}
\L_\gamma \eqdef \left\{ \varphi \in \Aut_{\C}(S)\, | \, \Ad(\varphi)(W)= W\right\}\, ,
\end{equation*}
where $W = \gamma(V)$.
\end{definition}

\noindent 
Notice that $\L_\gamma$ consists of $\C$-linear transformations of
$S$.

\begin{definition}
\label{def:adjointRep}
Given a weakly-faithful complex Clifford representation
$\gamma\in\Ob(\mClRep_w)$, we define the {\em adjoint representation}
$\Ad_{\gamma}\colon \L_\gamma \to \O(V,h)$ of $\L_\gamma$ by:
\begin{equation*}
\Ad_{\gamma}(\varphi) \eqdef (\gamma|_V)^{-1}\circ \Ad(\varphi) \circ (\gamma|_V)~,
\end{equation*}
for all $\varphi\in \L_\gamma$.
\end{definition}

\noindent
The following proposition partially characterizes the image of the
adjoint representation $\Ad_{\gamma}$ and in particular
shows that $\Ad_{\gamma}$ is well-defined.

\begin{prop}
Let $\gamma:\Cl(V,h)\rightarrow \End_\C(S)$ be a weakly-faithful
complex Clifford representation and $w$ be an element of the space
$W=\gamma(V)$. Then $w$ is an element of $\L_\gamma$ and $\Ad_{\gamma}(w) =
-R_{w_{0}}$, where $R_{w_{0}}$ denotes the orthogonal reflection of
$(V,h)$ with respect to the hyperplane orthogonal to the vector $w_{0} \eqdef
(\gamma|_V)^{-1}(w)\in V$. If $d\eqdef \dim_\R V$ is even, then we
have $\Ad_{\gamma}(\L_\gamma) = \O(V,h)$. If $d$ is odd, then we have
$\SO(V,h) \subseteq \Ad_{\gamma}(\L_\gamma)$.
\end{prop}

\begin{proof}
Explicit computation shows that:
\begin{equation*}
\Ad_{\gamma}(w)(v) = w_{0} v w_{0}^{-1} = - v + 2\,\frac{g(w_{0},v)}{g(w_{0},w_{0})} w_{0} = -R_{w_{0}}(v)\, ,
\end{equation*}
for every $v\in V$, where $w_{0} = \gamma^{-1}(w)\in V$. The
Cartan-Dieudonn\'e theorem implies that
$\Ad_{\gamma}(\L_\gamma) = \O(V,h)$ for even $d$ and $\SO(V,h)
\subseteq \Ad_{\gamma}(\L_\gamma)$ for odd $d$.
\end{proof}

\begin{remark}
For $d$ odd, the image $\Ad_{\gamma}(\L_\gamma)\subseteq \O(V,h)$ of
$\L_\gamma$ in $\O(V,h)$ depends on the details of the particular
weakly-faithful representation $\gamma$ under consideration. We will
see for instance in Proposition \ref{prop:Cliffordgrouphomotopy} that
if $\gamma$ is irreducible, then $\Ad_{\gamma}(\L_\gamma) =
\SO(V,h)$.
\end{remark}

\noindent
The following Proposition follows form Corollary \ref{cor:Homequiv},
but we give a direct proof because of its importance later.

\begin{prop}
Let $\gamma$ be a weakly-faithful complex Clifford
representation. Then the Lipschitz group $\L_\gamma$ is canonically
isomorphic with the automorphism group $\Aut_{\mClRep_w}(\gamma)$ of
$\gamma$ in the category $\mClRep_w$ of weakly faithful complex Clifford representations. In particular, the isomorphism
class of the group $\L_\gamma$ depends only on the isomorphism class of
$\gamma$ in that category.
\end{prop}

\begin{proof}
Any $\varphi\in \L_\gamma$ induces an invertible isometry $\varphi_0\in
\O(V,h)$ through relation \eqref{AdRelation}, namely:
\begin{equation}
\label{a0}
\varphi_0=(\gamma|_V)^{-1}\circ \Ad(\varphi)\circ (\gamma|_V)\in \O(V,h)\, ,
\end{equation}
which implies:
\begin{equation*}
\label{AdCl}
\gamma \circ \Cl(\varphi_0) =\Ad(\varphi)\circ \gamma\, .
\end{equation*}
Thus $(\varphi_0,\varphi)$ is the unique automorphism of $\gamma$ in
the category $\mClRep_w$ whose second component equals
$\varphi$. Conversely, we have $\varphi\in \L_\gamma$ for any
$(\varphi_0,\varphi)\in \Aut_{\mClRep_w}(\gamma)$ (see Proposition
\ref{prop:Ad}) and $\varphi_0$ is determined by $\varphi$ through
relation \eqref{a0} (see Proposition \ref{prop:AdRelation}). Hence the
map $F\colon \Aut_{\mClRep}(\gamma)\rightarrow \L_\gamma$ given by
$F(\varphi_0,\varphi)\eqdef \varphi$ is an isomorphism of groups which
allows us to identify $\L_\gamma$ with $\Aut_{\mClRep}(\gamma)$.
\end{proof}

\section{Elementary complex Lipschitz groups}

\begin{definition}
The complex Lipschitz group $\L_\gamma$ of an irreducible complex
Clifford representation is called an {\em elementary complex Lipschitz
  group}.
\end{definition}

Let $(V,h)$ be a real quadratic space of signature $(p,q)$ and
dimension $d=p+q$. When $d$ is even, all irreducible complex Clifford
representation $\gamma:\Cl(V,h)\rightarrow \End_\C(S)$ are mutually
$\C$-equivalent and faithful (these are called {\em Dirac
  representations} in \cite{FriedrichTrautman}). When $d$ is odd,
there exist up to $\C$-equivalence exactly two complex Clifford
representations $\gamma:\Cl(V,h)\rightarrow \End_\C(S)$, none of which
is faithful (these are called {\em Pauli representations} in
op. cit.). As we shall see below, both of these representations are
weakly-faithful. Moreover, they are unbasedly isomorphic in the
category $\mClRep_w$. In particular, the category $\mClRep_w$ contains
a single unbased isomorphism class of irreducible complex Clifford
representations of any given quadratic space $(V,h)$ (and this
isomorphism class is uniquely determined by the dimension and
signature of $(V,h)$).

\subsection{The case when $d$ is even}

In this case, there exists a single $\C$-equivalence class of
irreducible complex Clifford representations
$\gamma:\Cl(V,h)\rightarrow \End_\C(S)$ and any such representation is
faithful and satisfies $\dim_\C S=2^{\frac{d}{2}}$ (being a ``Dirac
representation''); in fact, the map $\gamma$ is bijective.  The
elementary complex Lipschitz group of such representations was
determined in \cite[Theorem 1]{FriedrichTrautman}:

\begin{prop}\cite{FriedrichTrautman}
When $d$ is even, we have $\L_\gamma=\Pin(V,h)\cdot \C^\times\simeq \Gamma(V,h)$,
which is homotopy-equivalent with $\Pin^{c}(V,h)$.
\end{prop}

\subsection{The case when $d$ is odd}

When $d$ is odd, an irreducible complex representation of $\Cl(V,h)$
is non-faithful of dimension $\dim_\C S=2^{\frac{d-1}{2}}$. Up to
$\C$-equivalence, there exist two such representations $\gamma^+$ and
$\gamma^-$ (called ``Pauli representations'' in
\cite{FriedrichTrautman}), which can be described as follows. Since
$d$ is odd, the Clifford volume element $\nu\in \Cl(V,h)$ determined
by some fixed orientation of $V$ is central in $\Cl(V,h)$ and
satisfies $\nu^2=\sigma_{p,q}\eqdef (-1)^{\frac{p-q-1}{2}}$.  Since
$\mCl(V,h)$ is generated by $V$ over $\C$, the real volume element is
also central in $\mCl(V,h)$, where it can be rescaled to a {\em
  complex volume element}\footnote{Notice that $\nu_\C$ belongs to
  $\Cl(V,h)$ iff $\sigma_{p,q}=+1$ (i.e. if and only if $p-q\equiv_8
  1,5$), in which case we have $\nu_\C=\pm \nu$.}  $\nu_\C\eqdef
\lambda \nu\in \mCl(V,h)$, where $\lambda$ is one of the two complex
square roots of $\sigma_{p,q}$. The complex volume element is central
in $\mCl(V,h)$ and satisfies $\nu^{2}_{\C} = 1$. Since $\nu_{\C}$ is
central and squares to $+1$, we can decompose $\mCl(V,h)$ as a direct
sum of two-sided ideals:
\begin{equation*}
\mCl(V,h) = \frac{1}{2}(1+\nu_{\C})\mCl(V,h)\oplus \frac{1}{2}(1-\nu_{\C})\mCl(V,h)\, ,
\end{equation*}
which are unital $\C$ algebras of dimension $2^{d-1}$, with units
given respectively by $\frac{1}{2}(1 + \nu_{\C})$ and $\frac{1}{2}(1 -
\nu_{\C})$ (note however that they are not unital subalgebras of
$\mCl(V,h)$). There exist two representations $\gamma_\C^+$ and
$\gamma_\C^-$ of $\mCl(V,h)$ which can be realized on the same
$\C$-vector space $S$ and are distinguished by their value on
$\nu_\C$:
\begin{equation*}
\gamma_\C^\pm(\nu_{\C}) = \pm \id_S~~.
\end{equation*}
The kernel of $\gamma_\C^\pm$ is given by:
\begin{equation}
\label{ker}
\ker\, \gamma_\C^\pm = \frac{1}{2}(1\mp \nu_{\C})\mCl(V,h)~~.
\end{equation}
Restricting $\gamma_\C^\pm$ gives unital isomorphisms of $\C$-algebras:
\begin{equation*}
\gamma_\C^\pm:\frac{1}{2}(1\pm \nu_{\C})\mCl(V,h) \stackrel{\sim}{\rightarrow}\End_{\C}(S)~~.
\end{equation*}
We have $\mCl(V,h)=\Cl(V,h)\oplus \i \Cl(V,h)$, where $\i$ denotes the
imaginary unit. The irreducible complex representations $\gamma^\pm$
are now obtained by restricting $\gamma_\C^\pm$ to $\Cl(V,h)$:
\be
\gamma^\pm\eqdef \gamma_\C^\pm|_{\Cl(V,h)}~~.
\ee

\begin{prop}
\label{prop:wfaithful}
The Pauli representations $\gamma^\pm$ are weakly-faithful and
isomorphic in the category $\mClRep_w$.
\end{prop}

\begin{proof}
The restriction $\gamma^\pm|_V=\gamma_\C^\pm|_V$ is injective iff
$V\cap \ker \gamma_\C =0$. The Chevalley-Crumeyrolle-Riesz isomorphism
identifies $\Cl(V,h)$ with $\wedge V^\vee$. Using this identification,
we have $x\,\nu_{\C} \in \wedge^{d-1} V^{\vee}_{\C}$ for any vector
$x\in V$. Relation \eqref{ker} shows that one can have $\ker
\gamma_\C^\pm\cap V\neq 0$ only when $d=2$, which is disallowed since
$d$ is odd. This shows that we must have $\ker\,\gamma_\C^\pm\cap V =
0$, which implies $\ker\,\gamma^\pm\cap V = 0$. This shows that
$\gamma^\pm$ are weakly-faithful.

It is well-known that $\gamma^-=\gamma^+\circ \pi$, where $\pi$ is the
parity involution of $\Cl(V,h)$.  Since $\pi=\Cl(-\id_V)$ is the
automorphism of $\Cl(V,h)$ which is induced by minus the identity map
of $V$ (which is an isometry of $V$), it follows that $\gamma^+$ and
$\gamma^-$ are isomorphic in the category $\mClRep_w$, being related
by the involutive unbased isomorphism:
\be
(-\id_V,\id_S):\gamma^+\stackrel{\sim}{\rightarrow} \gamma^-~~.
\ee
\end{proof} 

\begin{lemma}
\label{lemma:isoeven}	
When $d$ is odd, the restriction of $\gamma_\C^\pm$ gives a unital
isomorphism of associative $\C$-algebras from $\mCl^{+}(V,h)$ to
$\End_{\C}(S)$.
\end{lemma}

\begin{proof}
It is obvious that $\gamma_\C^\pm$ restricts to a unital morphism of
algebras from $\mCl^{+}(V,h)$ to $\End_\C(S)$. The equation
$x\,\nu_{\C} = \epsilon\, x$ has no solutions in $\mCl^{+}(V,h)$
because the left-hand side is odd while the right-hand side is even;
this shows that $\gamma_\C^\pm|_{\mCl^{+}(V,h)}$ is injective.  On the
other hand, we have $\dim_{\C}\, \mCl^{+}(V,h) =
\dim_{\C}\, \End_{\C}(S)=2^{d-1}$, so
$\gamma_\C^\pm|_{\mCl^{+}(V,h)}$ is bijective.
\end{proof}

\begin{prop}
\label{prop:Gammaodd}
Let $\gamma:\Cl(V,h)\rightarrow \End_\C(S)$ be an irreducible complex
Clifford representation such that $d=\dim_\R V$ is odd. Then
$\gamma_\C$ restricts to an isomorphism between
$\rGamma^{s}(V,h)\simeq \Spin^c(V,h)\cdot \C^\times$ and
$\L_\gamma$. In particular, the elementary complex Lipschitz group
$\L_\gamma$ is homotopy-equivalent with $\Spin^{c}(V,h)$.
\end{prop}

\begin{proof}
From Lemma \ref{lemma:isoeven} we have a unital isomorphism of
algebras $\gamma_{\C}:\mCl^{+}(V,h)\rightarrow \End_{\C}(S)$ which
restricts to an isomorphism of groups $\gamma_\C:
\mCl^{+}(V,h)^{\times}\rightarrow \Aut_{\C}(S)$.  In turn, this
further restricts to an isomorphism of groups:
\begin{equation*}
\gamma_\C\colon \Gamma^s(V,h) \xrightarrow{\sim} \L_\gamma\, .
\end{equation*}
The fact that $\L_\gamma$ is homotopy-equivalent to $\Spin^{c}(V,h)$
now follows from Proposition \ref{prop:Cliffordgrouphomotopy}.
\end{proof} 

Since for odd $d$ and complex irreducible $\gamma$ the kernel of
$\Ad_{\gamma}$ is given by the complex multiples of the
identity, the previous proposition gives the short exact sequence:
\begin{equation*}
1\to \C^\times \to \L_\gamma \xrightarrow{\Ad_{\gamma}} \SO(V,h) \to 1~~.
\end{equation*}

\begin{remark}
For odd $d$, reference \cite{FriedrichTrautman} determines the {\em
  non-elementary} complex Lipschitz group of the faithful but
reducible complex Cartan representation of $\Cl(V,h)$. As shown in
\cite{FriedrichTrautman}, that complex Lipschitz group is isomorphic
with $[\Pin(V,h)\rtimes (\C^\times)^2]/\Z_2$, where the semidirect
product in the numerator is determined explicitly in op. cit.
\end{remark}

\subsection{Realification of Clifford representations}
\label{subsec:realif}

Let $(V,h)$ be a real quadratic space of signature $(p,q)$ and
dimension $d=p+q$.  For $\K\in \{\R,\C\}$, let $\Rep_\K(\Cl(V,h))$
denote the ordinary category of $\K$-linear representations of
$\Cl(V,h)$ (whose morphisms are the {\em based} morphisms of
representations).  Let $\Rep_\K^w(\Cl(V,h))$ denote the full
sub-category of $\Rep_\K(\Cl(V,h))$ whose objects are the
weakly-faithful representations. Given a complex Clifford
representation $\gamma:\Cl(V,h)\rightarrow \End_\C(S)$, let $S_\R$ be
the underlying real vector space of $S$ and
$\gamma_\R:\Cl(V,h)\rightarrow \End_\R(S_\R)$ be the {\em
  realification} of $\gamma$, i.e. the real Clifford representation
obtained from $\gamma$ by forgetting the complex structure of
$S$. Given two complex Clifford representations
$\gamma_i:\Cl(V_i,h_i)\rightarrow \End_\C(S_i)$ (with $i=1,2$) and a
based morphism of complex representations $\varphi:\gamma_1\rightarrow
\gamma_2$, let $\varphi_\R:\gamma_{1\R}\rightarrow \gamma_{2\R}$ be
the based morphism of real representations obtained by forgetting the
complex structures of $S_1$ and $S_2$.

\begin{definition}
The {\em realification functor} $R:\Rep_\C(\Cl(V,h))\rightarrow
\Rep_\R(\Cl(V,h))$ is the functor which maps a complex representation
$\gamma:\Cl(V,h)\rightarrow \End_\C(S)$ to the real representation
$R(\gamma)\eqdef \gamma_\R:\Cl(V,h)\rightarrow \End_\R(S_\R)$ and a
based morphism $\varphi:\gamma_1\rightarrow \gamma_2$ of complex
representations to the based morphism of real representations
$\mathrm{R}(\varphi)\eqdef \varphi_\R:\gamma_{1\R}\rightarrow \gamma_{2\R}$.
\end{definition}

\

\noindent It is clear that $\mathrm{R}$ is faithful and that it maps
$\Rep_\C^w(\Cl(V,h))$ to $\Rep_\R^w(\Cl(V,h))$.

\begin{prop}
Let $\gamma:\Cl(V,h)\rightarrow \End_\C(S)$ be an irreducible complex
Clifford representation, where $(V,h)$ has signature $(p,q)$ and
dimension $d=p+q$.  Then the realification $\gamma_\R$ of $\gamma$ is
$\R$-irreducible iff $p-q\equiv_8 3,4,6,7$.
\end{prop}

\begin{proof}
Distinguish the cases:
\begin{enumerate}[1.]
\itemsep 0.0em
\item When $d$ is even. Then $\gamma$ is a Dirac representation and we
  have $\dim_\C S=2^{\frac{d}{2}}$. Thus $\dim_\R
  S_\R=2^{\frac{d}{2}+1}$. Comparing with Table 4 of
  \cite[Sec. 2.4]{fierz}, we see that $\gamma_\R$ is irreducible iff
  $p-q\equiv_9 4,6$ (which corresponds to the ``quaternionic simple
  case'' of op. cit.). 
\item When $d$ is odd. Then we can take $\gamma$ to be one of the
  Pauli representations $\gamma^\pm$, both of which satisfy $\dim_\C
  S=2^{\left[\frac{d}{2}\right]}$. Thus $\dim_\R
  S_\R=2^{\left[\frac{d}{2}\right]+1}$. Comparing with Table 4 of
  \cite[Sec. 2.4]{fierz}, we see that $\gamma_\R$ is irreducible iff
  $p-q\equiv_9 3,7$ (which corresponds to the ``complex case'' of
  op. cit.).
\end{enumerate}
\end{proof}

\noindent Let $\Irrep_\K(\Cl(V,h))$ denote the full subcategory of
$\Rep_\K^w(\Cl(V,h))$ whose objects are the irreducible
representations. 

\begin{prop}
When $p-q\equiv_8 3,4, 6, 7$, the restriction
$R:\Irrep_\C(\Cl(V,h))^\times \rightarrow \Irrep_\R(\Cl(V,h))^\times$
of the realification functor is faithful and strictly surjective but
not full. Moreover:
\begin{enumerate}[1.]
\itemsep 0.0em
\item When $p-q\equiv_8 3,7$, the $R$-preimage of a real irreducible
  representation $\eta:\Cl(V,h)\rightarrow \End_\R(\Sigma)$ has
  cardinality two and consists of the $\C$-inequivalent complex Pauli
  representations of $\Cl(V,h)$ whose representation space is obtained
  by endowing $\Sigma$ with the complex structures $J_\pm=\pm
  \gamma(\nu)$, where $\nu$ is the Clifford volume form determined by
  some orientation of $V$.
\item When $p-q\equiv_8 4,6$, the $R$-preimage of a real irreducible
  representation $\eta:\Cl(V,h)\rightarrow \End_\R(\Sigma)$ is in
  bijection with the unit sphere $\rS^2$ and consists of the complex
  Dirac representations whose complex structures are the unit
  imaginary elements of the commutant of $\im(\eta)$ (which in these
  cases is a quaternion algebra).
\end{enumerate}
\end{prop} 

\begin{proof}
The fact that $R$ is faithful follows immediately from its definition. 
For the remaining statements, consider the cases:
\begin{enumerate}[1.]
\itemsep 0.0em
\item $p-q\equiv_8 3,7$. In this case, there exist exactly two complex
  structures $J_\pm$ on the $\R$-vector space $\Sigma$ which lie in
  the commutant of $\im (\eta)$ (see \cite[Section
    2.7]{fierz}). Namely, one has $J_\pm=\pm \gamma(\nu)$, where
  $\nu\in \Cl(V,h)$ is the Clifford element defined by a fixed
  orientation of $V$. Since $J_\pm$ commute with $\eta(x)$ for all
  $x\in \Cl(V,h)$, it follows that $\eta$ coincides with the
  realification of any of the complex representations
  $\gamma^\pm:\Cl(V,h)\rightarrow \End_\C(S^\pm)$ obtained from $\eta$
  by endowing $\Sigma$ with the complex structure $J_\pm$ to obtain a
  complex vector space $S^\pm$. Notice that $S^-$ coincides with the
  complex-conjugate of $S^+$, hence $\gamma^+$ and $\gamma^-$ are
  mutually inequivalent Pauli representations (the dimension $d=p+q$
  of $V$ is odd when $p-q\equiv_8 3,7$). We have $\Sigma=(S^\pm)_\R$
  and $\eta=(\gamma_\pm)_\R$, which shows that the restriction of $R$
  is (strictly) surjective. Moreover, the $R$-preimage of $\eta$
  consists of the two complex representations $\gamma_+$ and
  $\gamma_-$.

Let $\eta':\Cl(V,h)\rightarrow \End_\R(\Sigma')$ is a real irreducible
representation which is equivalent with $\gamma$ and $J'_\pm \eqdef
\pm \gamma'(\nu)$. Given a based isomorphism (equivalence of
representations) $\psi:\eta\stackrel{\sim}{\rightarrow}\eta'$, we have
$\psi\in \Hom_\R(\Sigma,\Sigma')$ and:
\be
\psi\circ \eta(x)=\eta'(x)\circ \psi ~~~\forall x\in \Cl(V,h)~~.
\ee
Since $J_\pm$ lies in the commutant of $\im (\eta)$, this relation
implies that the complex structure $\psi\circ J_\pm\circ
\psi^{-1}\in \End_\R(\Sigma')$ lies in the commutant of $\im(\eta')$
and hence that we must have $\psi\circ J_\pm\circ
\psi^{-1}=\epsilon_\psi J'_\pm$, i.e. $\psi\circ J_\pm=\epsilon_\psi
J'_\pm\circ \psi$ for some $\epsilon_\psi\in \{-1,1\}$. This give the
disjoint union decomposition:
\be
\Iso_{\Rep_\R(\Cl(V,h))}(\eta,\eta')=\Iso^+_{\Rep_\R(\Cl(V,h))}(\eta,\eta')\sqcup \Iso^-_{\Rep_\R(\Cl(V,h))}(\eta,\eta')~~,
\ee
where: 
\be
\Iso^\pm_{\Rep_\R(\Cl(V,h))}(\eta,\eta')\eqdef \{\psi\in \Iso_{\Rep_\R(\Cl(V,h))}(\eta,\eta')|\epsilon_\psi=\pm 1\}~~.
\ee
It is clear from the above that:
\ben
\label{Rimcx}
R(\Iso_{\Rep_\C(\Cl(V,h))}(\gamma^{\epsilon},\gamma'^{\epsilon'}))=\Iso^{\epsilon\epsilon'}_{\Rep_\R(\Cl(V,h))}(\eta,\eta')~~~\forall \epsilon,\epsilon'\in \{-1,1\}~~.
\een
In particular, each of the sets
$\Iso^\pm_{\Rep_\R(\Cl(V,h))}(\eta,\eta')$ is non-empty and relation
\eqref{Rimcx} shows that $R$ is not full.
\item  $p-q\equiv_8 4,6$. In this case, the commutant of $\im(\eta)$
  inside the $\R$-algebra $\End_\R(\Sigma)$ (which is called the {\em
    Schur algebra} of $\eta$) is a unital associative algebra
  isomorphic with the algebra $\H$ of quaternions. The set of those
  complex structures on $\Sigma$ which lie in the commutant of
  $\im(\gamma)$ corresponds to the unit imaginary quaternions, being
  in bijection with the unit two-sphere $\rS^2$. It is clear that
  $\eta$ coincides with the realification of any of the complex Dirac
  representations obtained from $\eta$ by endowing $\Sigma$ with one
  of these complex structures. This immediately implies that $\mathrm{R}$ is
  not full.
\end{enumerate}
\end{proof}

\subsection{Comparison of complex and real elementary Lipschitz groups when $p-q\equiv_8 3,4,6,7$}

For $p-q\equiv_8 3,7$, the elementary complex Lipschitz group of an
irreducible complex Clifford representation $\gamma$ is
homotopy-equivalent to $\Spin^c(V,h)$, while the reduced elementary
real Lipschitz group associated to $\gamma_{\R}$ is
homotopy-equivalent to $\Spin^{o}(V,h)$ \cite{Lipschitz}. The latter
group is $\Z_2$-graded with even component equal to the $\Spin^c(V,h)$
subgroup consisting of those elements of $\Spin^o(V,h)$ which are
$\C$-linear automorphisms of $S$:
\begin{equation*}
\Spin^c(V,h)\simeq \{a\in \Spin^{o}(V,h) | a\in \End_\C(S)\}\, . 
\end{equation*}
Unlike $\Spin^{c}(V,h)$, the reduced elementary real Lipschitz group
$\Spin^{o}(V,h)$ also contains $\R$-linear endomorphisms of $S$ which
are $\C$-antilinear; these form the odd component of $\Spin^{o}(V,h)$.

For $p-q\equiv_8 4,6$, the elementary complex Lipschitz group of an
irreducible complex Clifford representation $\gamma$ is
homotopy-equivalent to $\Pin^c(V,h)$, while the reduced elementary
real Lipschitz group associated to $\gamma_{\R}$ is
homotopy-equivalent to $\Pin^q(V,h)$ \cite{Lipschitz}.

\begin{remark} 
When $p-q\not\equiv_8 3,4, 6, 7$, the real Clifford representation
obtained by realification from an irreducible complex Clifford
representation is reducible. In that case, there is no simple relation
between real and complex elementary Lipschitz groups.
\end{remark}


\section{Complex pinor bundles and complex Lipschitz structures}
\label{sec:complexpinorbundles}


\subsection{Complex Lipschitz structures}

Let $(V,h)$ be a quadratic real vector space and $\eta\colon
\Cl(V,h)\to \End_{\mathbb{C}}(\Sigma)$ be a weakly-faithful complex
Clifford representation, where $\Sigma$ is a $\C$-vector space. We
denote by $\L_{\eta}$ the corresponding complex Lipschitz group and by
$\Ad_{\eta} \colon \L_\eta \to \O(V,h)$ its adjoint representation.

\begin{definition}
\label{def:spinostructures}
Let $M$ be a connected manifold and $P_{\O}$ be a
principal $\O(V,h)$-bundle over $M$. A {\em complex Lipschitz
  structure} of {\em type} $\eta$ on $P_{\O}$ is a pair $(Q,
\Lambda)$, where $Q$ is a principal $\L_{\eta}$-bundle over $M$ and
$\Lambda\colon Q\rightarrow P_{\O}$ is a bundle map fitting into the
following commutative diagram:
\begin{equation*}
\scalebox{1.0}{
\xymatrix{
Q\times \L_\eta\,\ar[d]_{\Lambda\times \Ad_\eta}\ar[r] ~ &~ Q\ar[d]^{\Lambda}\\
P_\O\times \O(V,h)\ar[r] & ~ P_\O\\
}}
\end{equation*}
where the horizontal arrows denote the right action of the group on
the corresponding bundle.
\end{definition}

\begin{definition}
\label{def:Lmf}
Let $(Q^1,\Lambda^1)$ and $(Q^2,\Lambda^2)$ be two complex Lipschitz
structures of type $\eta$ over $P_{\O}$. A {\em morphism of complex
  Lipschitz structures} from $(Q^1,\Lambda^1)$ to $(Q^2,\Lambda^2)$ is
a morphism $F\colon Q^{1}\to Q^{2}$ of principal $\L_{\eta}$-bundles
such that $\Lambda^{2}\circ F= \Lambda^{1}$, i.e. such that the
following diagram commutes:
\begin{equation*}
\scalebox{1.0}{
\xymatrix{
Q^{1}\,\ar[d]_{\Lambda^{1}}\ar[r]^{F} ~ & ~ Q^{2} \ar[d]^{\Lambda^{2}}\\
P_{\O} \ar[r]^{\mathrm{id}} ~ & ~ P_{\O}\\
}}
\end{equation*}
\end{definition}

\noindent
Let $(M,g)$ be a connected pseudo-Riemannian manifold. In the
following we take $P_{\O}$ to be the orthogonal coframe bundle
$P_{\O}(M,g)$ of $(M,g)$, so $(V,h)$ is an isometric model of any of
the quadratic spaces $(T_p^\ast M, g_p^\ast)$, where $p$ is a point of
$M$ and $g_p^\ast$ is the contragradient metric induced by $g$ on
$T_p^\ast M$. A complex Lipschitz structure on $P_\O(M,g)$ be called a
{\em complex Lipschitz structure on $(M,g)$}.
 
\begin{definition}
Let $\L_{\eta}(M,g)$ be the category whose objects are the
complex Lipschitz structures of type $\eta$ on $(M,g)$ and whose
arrows are morphisms of complex Lipschitz structures. We denote by
$\L_{\eta}(M,g)^\times$ the corresponding groupoid.
\end{definition}

\subsection{Reduced complex Lipschitz structures}

Let $\eta:\Cl(V,h)\rightarrow \End_\C(\Sigma)$ be a weakly-faithful
Clifford representation and $\L_\eta$ be the corresponding complex
Lipschitz group.

\begin{definition}
A {\em reduction} of $\L_\eta$ is a closed subgroup $\L^0_\eta\subset
\L_\eta$ such that there exists a surjective morphism of Lie groups
$r:\L_\eta\rightarrow \L^0_\eta$ which satisfies the conditions:
\begin{enumerate}
\item $r$ is a homotopy retraction of the inclusion map
  $\iota:\L^0_\eta\hookrightarrow \L_\eta$, i.e. we have $r\circ\iota
  = \id_{\L^0_{\eta}}$ and the map $\iota\circ r:\L_\eta\rightarrow
  \L_\eta$ is homotopic to $\id_{\L_\eta}$.
\item We have $\Ad_\eta\circ r=\Ad_\eta$, i.e. $r$ fits into the
  following commutative diagram:
\begin{equation*}
\xymatrix{
~& \L^0_{\eta} \ar@/^0.6pc/[dl]^{\iota} \ar[dr]^{\Ad_{\eta}\circ \iota} & \\
\L_{\eta} \ar[ur]^{r}  \ar[rr]_{\Ad_{\eta}} & & \O(V,h) 
}
\end{equation*}
\end{enumerate}
In this case, $r$ is called an {\em equivariant homotopy retraction}
of $\L_\eta$ onto $\L_\eta^0$.
\end{definition}

\begin{remark}
Since $r:\L_\eta\rightarrow \L_\eta^0$ is a retraction, the Lie group
$K_\eta\eqdef \ker r=r^{-1}(1)$ is contractible (indeed, the
restriction $r|_{K_\eta}:K_\eta\rightarrow 1$ is a homotopy retraction).
\end{remark}

\begin{example}
When $\dim_\R V$ is even, the group $\Pin^c(V,h)$ is a reduction of
the elementary complex Lipschitz group $\Pin(V,h)\cdot \C^\times
\simeq \rGamma(V,h)$, an equivariant homotopy retraction being given
in equation \eqref{r}. When $\dim_\R V$ is odd, the group
$\Spin^c(V,h)$ is a reduction of the elementary complex Lipschitz
group $\Spin(V,h)\cdot \C^\times \simeq \rGamma^s(V,h)$, an
equivariant homotopy retraction being given by $r|_{\Spin^c(V,h)}$.
\end{example}

Let $\L_\eta^0$ be a reduction of $\L_\eta$, with inclusion map
$\iota$ and equivariant retraction map $r$. Let $\Ad_\eta^0\eqdef
\Ad_\eta|_{\L_\eta^0}$ and $K_\eta\eqdef \ker r$. A {\em reduced
  complex Lipschitz structure} with structure group $\L_\eta^0$ is
defined as in Definition \ref{def:spinostructures}, but replacing
$\L_\eta$ with $\L_\eta^0$ and $\Ad_\eta$ with $\Ad_\eta^0$. A
morphism of reduced complex Lipschitz structures is defined similarly
to Definition \ref{def:Lmf}. With these definitions, reduced complex
Lipschitz structures with structure group $\L_\eta^0$ and defined on
$(M,g)$ form a category denoted $\L_\eta^0(M,g)$, whose unit groupoid
we denote by $\L_\eta^0(M,g)^\times$. 

The right-split short exact sequence of groups: 
\be
\xymatrix{ 
1 \ar[r] & K_\eta  \ar[r] & \L_\eta \ar[r]^r & \ar@/^0.6pc/[l]^{\iota}\L_\eta^0 \ar[r] & 1\\ 
}
\ee
induces mutually inverse isomorphisms in cohomology:
\be
\xymatrix{ 
H^1(M,\L_\eta) \ar[r]^{r_\ast} & \ar@/^0.6pc/[l]^{\iota_\ast}H^1(M,\L_\eta^0) \\ 
}~~,
\ee
where we used the fact that $K_\eta$ is contractible.  This shows that
$r_\ast$ and $\iota_\ast$ give inverse bijections between the sets of
isomorphism classes of principal $\L_\eta$-bundles and principal
$\L_\eta^0$-bundles defined on $M$.  The maps $r_\ast$ and
$\iota_\ast$ are induced by the associated fiber bundle construction,
which gives mutually quasi-inverse functors between the corresponding
categories of principal bundles. Below, we show that these functors in
turn induce mutually quasi-inverse equivalences between the categories
$\L_\eta(M,g)$ and $\L_\eta^0(M,g)$.

\begin{definition}
Let $(Q,\Lambda)$ be a complex Lipschitz structure of type $\eta$ over
$(M,g)$. An \emph{$\L_\eta^0$-reduction of $(Q,\Lambda)$} is a triple
$(Q_{0},\Lambda_{0}, I)$, where $(Q_0,\Lambda_0)\in
\Ob(\L_\eta^0(M,g))$ and $I:Q_0\rightarrow Q$ is an injective fiber
bundle map such that the following conditions are satisfied:
\begin{enumerate}[1.]
\item $I$ is an $\L^0_{\eta}$-reduction of the principal bundle
  $Q$. Thus $I(q_0 u_0) = I(q_0)\iota(u_0)$ for all $q_0\in Q_0$ and all
  $u_0\in \L^0_{\eta}$.
\item The following diagram commutes:
\begin{equation*}
\scalebox{1.0}{
\xymatrix{
Q_0\,\ar[dr]_{\Lambda_0}\ar[r]^{I} ~ & ~ Q \ar[d]^{\Lambda}\\
& P_{\O}(M,g)\\
}}
\end{equation*}		
\end{enumerate}
\end{definition}

\begin{definition}
Let $(Q,\Lambda)$ be a complex Lipschitz structure of type $\eta$ over
$(M,g)$. An \emph{$\L_\eta^0$-retraction of $(Q,\Lambda)$} along $r$
is a triple $(Q_0,\Lambda_0, R)$, where $(Q_0,\Lambda_0)\in
\Ob(\L^0_\eta(M,g))$ and $R:Q\rightarrow Q_0$ is a surjective fiber
bundle map such that the following conditions are satisfied:
\begin{enumerate}[1.]
\item $R(q u) = R(q) r(u)$ for all $q\in Q$ and all
  $u\in \L_{\eta}$.
\item The following diagram commutes:
\begin{equation*}
\scalebox{1.0}{
\xymatrix{
Q\,\ar[dr]_{\Lambda}\ar[r]^{R} ~ & ~ Q_0 \ar[d]^{\Lambda_0}\\
& P_{\O}(M,g)\\
}}
\end{equation*}		
\end{enumerate}
\end{definition}

\begin{remark}
Given an $\L_\eta^0$-retraction $(Q_0,\Lambda_0, R)$ of $(Q,\Lambda)$
along $r$, the map $R$ makes $Q$ into a principal fiber bundle over
$Q_0$ with structure group $K_\eta=\ker r$. This principal bundle is
trivial since $K_\eta$ is contractible (which implies that the bundle
has a section).
\end{remark}

\begin{prop}
\label{prop:cI}
Consider the functor $\cI:\L^0_{\eta}(M,g) \rightarrow \L_{\eta}(M,g)$
defined as follows: 
\begin{itemize}
\itemsep 0.0em
\item For any $(Q_{0},\Lambda_{0})\in \Ob(\L^0_{\eta}(M,g))$, define
  $(Q,\Lambda):=\cI(Q_0,\Lambda_0)\in \Ob(\L_\eta(M,g))$ through:
\begin{equation*}
Q \eqdef Q_0\times_{\iota}\L_{\eta}~~,~~\Lambda([q_0, u]_\iota)\eqdef \Lambda_0(q_0)\circ \Ad_{\eta}(u)~~~~~
\forall [q_0,u]_\iota\in Q~.
\end{equation*}
\item For any morphism $F_0 \colon (Q^1_{0},\Lambda^1_{0})\to
  (Q^2_{0},\Lambda^2_{0})$ in $\L^0_{\eta}(M,g)$, define
  $F:=\cI(F_0)$ through:
\begin{equation*}
F([q_0,u]_\iota) \eqdef [F_0(q_0),u]_\iota~~~\forall\,\, [q_0,u]_\iota\in \cI(Q^1_0)\, .
\end{equation*}
\end{itemize}
Then $(Q_0,\Lambda_0,I)$ is an $\L_\eta^0$-reduction of $(Q,\Lambda)$,
where the injective fiber bundle map $I:Q_0\rightarrow Q$ is given by:
\ben
\label{Idef}
I(q_0)\eqdef [q_0,1]_\iota~~\forall q_0\in Q_0~~.
\een
Moreover, the surjective map $\pi:Q\rightarrow Q_0$ defined through:
\be
\pi([q_0,u]_\iota)=q_0 r(u)
\ee
satisfies $\pi\circ I=\id_{Q_0}$ and makes $(Q_0, \Lambda_0, \pi)$
into an $\L_\eta^0$-retraction of $(Q,\Lambda)$ along $r$.
\end{prop}

\begin{proof}
The morphism $\Lambda$ makes $Q$ into a Lipschitz structure since
$\Lambda([q_0,u]\cdot w) = \Lambda_{0}(q_0)\circ \Ad_{\eta}(u w)
=\Lambda([q_0,u])\circ\Ad_{\eta}(w)$ for all $w\in\L_{\eta}$. Given
$[q_0,u]_\iota\in Q^1$ and $u_0\in\L^0_{\eta}$, we have:
\begin{equation*}
F([q_0 u_0, u^{-1}_0 u]_\iota) = [F_0 (q_0 u_0), u^{-1}_0 u]_\iota = [F_0 (q_0)
  u_0, u^{-1}_0 u]_\iota = [F(q_0),u]_\iota=F([q_0,u]_\iota)~~,
\end{equation*}
which shows that $F$ is well-defined. The fact that
$(Q_0,\Lambda_0,I)$ is an $\L_\eta^0$-reduction of $(Q,\Lambda)$
follows from:
\begin{equation*}
(\Lambda\circ I)(q_0) = \Lambda([q_0,1]_{\iota}) = \Lambda_{0}(q_0)\, , \qquad \forall\,\, q_0 \in \L^0_{\eta}~~.
\end{equation*}
Any element $q\in Q$ can be written as $q=[q_0,x]_\iota$, where $x\in
\ker r$ and $q_0:=\pi_0(q)\in Q_0$ are uniquely determined by
$q$. Indeed, any $u\in \L_\eta$ can be written uniquely as $u=u_0 x$
with $u_0\in \L_\eta^0$ and $x\in K_\eta=\ker r$, namely $u_0= r(u)$
and $x=u_0^{-1}u$. 

The map $\pi$ is well-defined since: 
\be
\pi([q_0 u_0, u]_\iota)=q_0 u_0 r(u)=q_0 r(u_0 u)~~,
\ee
where we used the fact that $r|_{\L_\eta^0}=\id_{\L_\eta^0}$ and that
fact that $r$ is a morphism of groups.  We have
$\pi^{-1}(q_0)=\{[q_0,x]_\iota|x\in K_\eta\}$ and $\pi(qu)=\pi(q)r(u)$
for all $q\in Q$ and $u\in \L_\eta$.  Moreover, we have:
\be
\Lambda_0(\pi([q_0,u]_\iota))=\Lambda_0(q_0r(u))=\Lambda_0(q_0)\Ad_\eta(r(u))=\Lambda_0(q_0)\Ad_\eta(u)=\Lambda([q_0,u]_\iota)~~,
\ee
where we used the relation $\Ad_\eta\circ r=\Ad_\eta$.  This shows
that $(Q_0, \Lambda_0,\pi)$ is an $\L_\eta^0$-retraction of
$(Q,\Lambda)$ along $r$. It is clear that we have $\pi\circ
I=\id_{Q_0}$.
\end{proof}

\begin{prop}
\label{prop:cR}
Consider the functor $\cR:\L_{\eta}(M,g) \rightarrow \L_{\eta}^0(M,g)$
defined as follows:
\begin{itemize}
\itemsep 0.0em
\item For any $(Q,\Lambda)\in \Ob(\L_{\eta}(M,g))$, define
  $(Q_0,\Lambda_0):=\cR(Q,\Lambda)\in \Ob(\L_\eta^0(M,g))$ through:
\begin{equation*}
  Q_0 \eqdef Q\times_{r}\L^0_{\eta}~~,~~\Lambda_0([q, u_0]_r)\eqdef \Lambda(q)\circ \Ad_{\eta}^0(u_0)~~~ \forall [q,u_0]_r\in Q\times_r \L_\eta^0
\end{equation*}
\item For any morphism $F\colon (Q^1,\Lambda^1)\to
  (Q^2,\Lambda^2)$ in $\L_{\eta}(M,g)$, define
  $F_0:=\cR(F)$ through:
\begin{equation*}
F_0([q,u_0]_r) \eqdef [F(q),u_0]_r~~~\forall\,\, [q,u_0]_r\in \cR(Q^1)\, .
\end{equation*}
\end{itemize}
Then $(Q_0,\Lambda_0,R)$ is an $\L_\eta^0$-retraction of $(Q,\Lambda)$
along $r$, where the surjective fiber bundle map $R:Q\rightarrow Q_0$
is given by:
\ben
\label{Rdef}
R(q)\eqdef [q,1]_r ~~\forall q\in Q~~.
\een
Moreover, there exists an injective fiber bundle map $j:Q_0\rightarrow
Q$ which makes $(Q_0,\Lambda_0, j)$ into an $\L_\eta^0$-reduction of
$(Q,\Lambda)$ and satisfies $R\circ j=\id_{Q_0}$.
\end{prop}

\begin{proof}
Clearly $Q_0$ is a principal $\L^0_{\eta}$-bundle. To see that
$\Lambda_0$ is well-defined, note that for all $u\in \L_{\eta}$ we
have:
\begin{equation*}
\Lambda_{0}([q u, r(u)^{-1} u_0]_{r}) = \Lambda(q u)\circ \Ad^0_{\eta}(r(u)^{-1} u_0) = 
\Lambda(q)\circ \Ad_{\eta}(u)\circ \Ad_\eta( r(u))^{-1}\circ \Ad_{\eta}(u_0) = \Lambda_{0}([q, u_0]_{r})\, ,
\end{equation*}
where we used the relation $\Ad_\eta\circ r=\Ad_\eta$. 
The fact that $F_0$ is well-defined follows from direct computation by
using the equivariance of $F$. To show
that $R$ is an $\L_\eta^0$-retraction of $(Q,\Lambda)$ along $r$, we
compute:
\begin{equation*}
(\Lambda_{0}\circ R)(q) = \Lambda_0([q,1]_{r}) = \Lambda(q)\Ad_\eta^0(1)=\Lambda(q)~~.
\end{equation*}
Since $r:\L_\eta\rightarrow \L_\eta^0$ is surjective, any $q_0\in Q_0$
can be written in the form $q_0=[q,1]_r$, where $q\in Q$ is determined
by $q_0$ up to a transformation of the form $q\rightarrow q x$, where
$x\in \ker r$. Hence $R^{-1}(q)$ is a $\ker r$-torsor and the map
$R:Q\rightarrow Q_0$ presents $Q$ as a principal $\ker r$-bundle over
$Q_0$. Hence the structure group of the principal bundle $R:Q\rightarrow Q_0$
is contractible, which implies that this principal bundle is trivial
and thus has a section. We claim that we can find a section
$j:Q_0\rightarrow Q$ of $R$ which is $\L_\eta^0$-equivariant. 

Let $\Diff(K_\eta)$ denote the group of diffeomorphisms of $K_\eta$. 
The map $\varphi:\L_\eta\rightarrow \Diff(K_\eta)$ defined through:
\be
\varphi(u)(x)\eqdef r(u) x u^{-1}~~\forall u\in \L_\eta~\forall x\in K_\eta
\ee
is a morphism of groups which defines a smooth left action of the Lie group 
$\L_\eta$ on the underlying manifold of $K_\eta$. Consider the
fiber bundle defined through $T\eqdef Q\times_\varphi K_\eta$. 

We next show existence\footnote{This argument is based on a suggestion
  of A. Moroianu.} of an $\L_\eta^0$-equivariant section of
$R$.  Since $K_\eta$ is contractible, the fiber bundle $T$ admits a
section $s\in \Gamma(M,T)$. This section corresponds to a map
$\sigma:Q\rightarrow K_\eta$ such that
$\sigma(qu)=\varphi(u)^{-1}\sigma(q)$, i.e.:
\ben
\label{sigma}
\sigma(qu)=r(u)^{-1}\sigma(q) u~~\forall q\in Q~~\forall u\in \L_\eta~~.
\een
Let $j:Q_0\rightarrow Q$ be the map defined through:
\be
j([q,u_0]_r)\eqdef q \sigma(q)^{-1} u_0~~,~~\forall q\in Q~~\forall u_0\in \L_\eta^0~~.
\ee
Then \eqref{sigma} implies that $j$ is well-defined, since for all $u\in \L_\eta$ we have:
\be
j([qu,u_0]_r)=qu \sigma(qu)^{-1}u_0=qu u^{-1}\sigma(q)^{-1} r(u)u_0=q \sigma(q)^{-1}r(u) u_0=j([q,r(u)u_0]_r)~~.
\ee 
Moreover, we have:
\be
(R\circ j)([q,u_0]_r)=R(q \sigma(q)^{-1} u_0)=R(q) r(\sigma(q)^{-1}u_0)=R(q)r(\sigma(q)^{-1})r(u_0)=R(q)u_0=[q,1]_ru_0=[q,u_0]_r~~,
\ee
where we used the fact that $\sigma(q)^{-1}\in K_\eta=\ker r$ and the identity
$r|_{\L_\eta^0}=\id_{\L_\eta^0}$. Thus $R\circ j=\id_{Q_0}$. We also have:
\be
j([q,u_0]_rv_0)=j([q,u_0v_0]_r)=j([q,u_0]_r)v_0~~\forall v_0\in \L_\eta^0~~.
\ee
Hence $j$ is an $\L_\eta^0$-equivariant section of $R:Q\rightarrow Q_0$:
\ben
\label{jequiv}
j(q_0 u_0)=j(q_0) u_0~~\forall q_0\in Q_0~~\forall u_0\in \L_\eta^0~~. 
\een
\end{proof}

\begin{thm}
\label{thm:Lredequiv}
The functors $\cI$ and $\cR$ give mutually quasi inverse equivalences
between the categories $\L_\eta^0(M,g)$ and $\L_\eta(M,g)$.
\end{thm}

\begin{proof}
Composition of $\cI$ and $\cR$ gives functors:
\begin{equation*}
\cI\circ \cR:\L_{\eta}(M,g) \rightarrow \L_{\eta}(M,g) ~~\mathrm{and}~~\cR\circ \cI:\L_\eta^0(M,g) \to \L^0_{\eta}(M,g)~.
\end{equation*}
We will construct isomorphisms of functors: 
\be
\cN:\id_{\L_\eta(M,g)}\stackrel{\sim}{\rightarrow} \cI\circ \cR~~\mathrm{and}~~\cK:\id_{\L_\eta^0(M,g)}\stackrel{\sim}{\rightarrow} \cR\circ \cI~~.
\ee
\begin{enumerate}[1.]
\itemsep 0.0em
\item Construction of $\cN$. Let $(Q,\Lambda)$ be a complex Lipschitz
  structure of type $\eta$ and set $(Q_0,\Lambda_0)\eqdef
  \cR(Q,\Lambda)$. Let $R:Q\rightarrow Q_0$ be the map which makes
  $(Q_0,\Lambda_0,R)$ into an $\L_\eta^0$-retraction of $(Q,\Lambda)$
  along $r$ and let $j:Q_0\rightarrow Q$ be an $\L_\eta^0$-equivariant
  section of $R$ which makes $(Q_0,\Lambda_0,j)$ into an
  $\L_\eta^0$-reduction of $(Q,\Lambda)$ (see Proposition
  \ref{prop:cR}). For every $q\in Q$, let $x_q$ be the unique element
  of $\L_{\eta}$ such that $q = j(R(q)) x_q$. For any $u\in
  \L_{\eta}$, direct computation using $\L_\eta^0$-equivariance of $j$
  shows that $x_{q u} =r(u)^{-1}\, x_{q}\, u$. Applying $R$ to the
  relation $q =j( R(q)) x_{q}$ gives:
\ben
\label{Rj}
R(q) = (R\circ j \circ R)(q) r(x_{q}) = R(q) r(x_{q})\, ,
\een
where we used the identity $R\circ j=\id_{Q_0}$. Since the action of
$\L_{\eta}$ on $Q$ is free, relation \eqref{Rj} gives
$r(x_{q})=1$. Let $(Q',\Lambda')\eqdef (\cI\circ
\cR)(Q,\Lambda)=\cI(Q_0,\Lambda_0)$ and consider the bijective fiber
bundle map:
\be
\cN_{Q,\Lambda}:Q\rightarrow Q'
\ee
given by $\cN_{Q,\Lambda}(q)\eqdef [[q,1]_r, x_{q}]_{\iota}$. For any 
$u\in \L_{\eta}$, we have:
\begin{equation*}
\cN_{Q,\Lambda}(q u) = [[q u,1]_r, r(u)^{-1}\, x_{q}\, u]_{\iota}  = [[q ,1]_r r(u), r(u)^{-1}\, x_{q}\, u]_{\iota}  = [[q ,1]_r, x_{q}\, u]_{\iota}  = \cN_{Q,\Lambda}(q) u\, .
\end{equation*}
Hence $\cN_{Q,\Lambda}$ is an isomorphism of principal
$\L_{\eta}$-bundles. Moreover, we have:
\begin{equation*}
\Lambda'\circ \cN_{Q,\Lambda}(q) = \Lambda'([[q,1]_r, x_{q}]_{\iota}) = \Lambda_0([q,1]_r)\circ \Ad_{\eta} (x_{q}) = \Lambda_0([q,1]_r) = \Lambda(q)~~,
\end{equation*}
where we used the fact that
$\Ad_\eta(x_q)=\Ad_\eta(r(x_q))=\Ad_\eta(1)=1$. Hence
$\cN_{Q,\Lambda}$ is an isomorphism of Lipschitz structures from
$(Q,\Lambda)$ to $(Q',\Lambda')=(\cI\circ \cR)(Q,\Lambda)$. Given any
morphism of Lipschitz structures $F:(Q^1,\Lambda^1) \rightarrow
(Q^2,\Lambda^2)$, it is easy to see that the following diagram
commutes:
\begin{equation*}
\scalebox{1.1}{
\xymatrix{
(\cI\circ \cR)(Q^1,\Lambda^1) \ar[r]^{(\cI\circ \cR)(F)} ~ & ~  (\cI\circ \cR)(Q^{2},\Lambda^{2})\\
(Q^1,\Lambda^1),\ar[u]_{\cN_{Q^1,\Lambda^1}} \ar[r]^{F} ~ & ~ (Q^2,\Lambda^2)\ar[u]_{\cN_{Q^2,\Lambda^2}}\\
}}
\end{equation*}
showing that $\cN$ is a natural transformation. 
\item Construction of $\cK$. For any $(Q_0,\Lambda_0)\in
\Ob(\L^0_\eta(M,g))$, there exists an isomorphism in $\L_\eta^0(M,g)$:
\begin{equation*}
\cK_{Q_0,\Lambda_0}:(Q_0,\Lambda_0)\xrightarrow{\sim} (\cR\circ \cI)(Q_0,\Lambda_0)~~.
\end{equation*}
given by $\cK_{Q_0,\Lambda_0}(q_0)\eqdef [[q_0,1]_\iota,1]_r=(R\circ
I)(q_0)$.  It is easy to see $\cK:\id_{\L_\eta^0(M,g)}\rightarrow
\cR\circ \cI$ is an invertible natural transformation.
\end{enumerate}
\end{proof}

\subsection{Complex pinor bundles}

\begin{definition}
A {\em complex pinor bundle} over $(M,g)$ is a pair $(S,\gamma)$,
where $S$ is a complex vector bundle over $M$ and $\gamma\colon
\Cl(M,g) \to End_{\mathbb{C}}(S)$ is a unital morphism of bundles of
algebras such that $\gamma_{p}\colon \Cl(T^{\ast}_{p} M,g^{\ast}_{p})
\to \End_{\mathbb{C}}(S_{p})$ is a weakly-faithful complex
representation of $\Cl(T^{\ast}_{p} M,g^{\ast}_{p})$ for any $p\in
M$. We say that $(S, \gamma)$ is of {\em type} $\eta$ if, for every $p\in
M$, the complex Clifford representations $\eta$ and $\gamma_{p}$ are
(unbasedly) isomorphic. We say that $(S, \gamma)$ is {\em elementary}
if its type is an irreducible complex Clifford representation.
\end{definition}

\begin{remark}
A complex pinor bundle is the same as a bundle of complex Clifford
modules defined on $(M,g)$. We use the name ``pinor'' (rather than
``spinor'') in order to distinguish such bundles from bundles of modules
defined over the even sub-bundle of the Clifford bundle.  Notice that
the type $\eta$ of a complex pinor bundle $(S,\gamma)$ is well-defined
up to (unbased) isomorphism of complex Clifford representations.
\end{remark}

\begin{definition}
A {\em based morphism of complex pinor bundles} $F\colon (S,\gamma)
\rightarrow (S^{\prime},\gamma^{\prime})$ is a based morphism $F\colon
S\rightarrow S^{\prime}$ of complex vector bundles such that:
\begin{equation*}
L_F\circ \gamma = R_F\circ \gamma^{\prime}\, ,
\end{equation*}
i.e. such that the fiber map $F_p\colon S_p\rightarrow S^{\prime}_p$
at any point $p\in M$ is a {\em based} morphism of Clifford
representations from $\gamma_p:\Cl(T_p^\ast
M,g_p^\ast)\rightarrow \End_\C(S_p)$ to
$\gamma^{\prime}_p:\Cl(T_p^\ast M,g_p^\ast)\rightarrow \End_\C(S'_p)$.
\end{definition}

\noindent 
Since $M$ is connected by assumption, all quadratic spaces $(T_p^\ast
M, g_p^\ast)$ are mutually isometric and isometric to some model
quadratic space $(V,h)$. Similarly, all fibers of $S$ are isomorphic
as $\C$-vector spaces and hence isomorphic with some model vector
space $\Sigma$. Using a common trivializing cover of $TM$ and $S$,
this implies that the complex Clifford representations
$\gamma_p:\Cl(T_p^\ast M, g_p^\ast) \rightarrow \End_\C(S_{p})$ ($p\in
M$) are mutually isomorphic in the category $\mClRep_w$ and hence
isomorphic with some model weakly-faithful representation
$\eta:\Cl(V,h)\rightarrow \End_\C(\Sigma)$, which defines the type of
$(S,\gamma)$. The isomorphism class of $\eta$ in the category
$\mClRep_w$ is invariant under isomorphism of complex pinor bundles.

\begin{definition}
Let $\mClB_{\eta}(M,g)$ be the category whose objects are
complex pinor bundles of type $\eta$ and whose arrows are based
morphisms of complex pinor bundles. We denote by
$\mClB_{\eta}(M,g)^\times$ the corresponding groupoid.
\end{definition}

\subsection{Relation between complex pinor bundles and complex Lipschitz structures}

\noindent The following proposition is the analogue of \cite[Proposition
  6.1]{Lipschitz} for complex pinor bundles. 

\begin{prop}
\label{propdef:functors}
Let $\eta\in \Ob(\mClRep_w)$. Consider the functors
$Q_\eta:\mClB_{\eta}(M,g)^\times\rightarrow \L_\eta(M,g)^\times$
and $S_\eta\colon \L_\eta(M,g)^\times\rightarrow
\mClB_{\eta}(M,g)^\times$ defined as follows:
\begin{enumerate}[A.]
\itemsep 0.0em
\item Let $(S,\gamma)$ denote a complex pinor bundle of type $\eta$ on
  $(M,g)$. Let $Q:=Q_\eta(S,\gamma)$ denote the principal bundle with
  structure group $\L:=\L_{\eta} = \Aut_{\mClRep}(\eta)$, total space:
\begin{equation*}
Q\eqdef \sqcup_{p\in M}\Hom_{\mClRep^\times_w}(\eta,\gamma_p)\, ,
\end{equation*}
projection given by $\pi(q)=p$ for $q \in Q_p=
\Hom_{\mClRep^\times_w}(\eta,\gamma_p)$ and right $\L$-action given by
$q\cdot g\eqdef q\circ g$ for all $g\in \L$. We topologize $Q$ in the
obvious way. Let $\Lambda \eqdef \Lambda_\eta(S,\gamma)\colon
Q_\eta(S,\gamma)\rightarrow P_{\O}(M,g)$ be the map defined through:
\begin{equation}
\label{eq:deftau0}
\Lambda_p(q) \eqdef q_0\in \Hom_{\Quad^\times}((V,h),(T_p^\ast M, g_p^\ast))=P_{\O}(M,g)_p~.
\end{equation}
Then, $(Q,\Lambda)$ is a Lipschitz structure on $(M,g)$ relative to
$\eta$, which we call the \emph{Lipschitz structure induced by
  $(S,\gamma)$}. A based isomorphism of complex pinor bundles $F
\colon (S,\gamma)\rightarrow (S^{\prime},\gamma^{\prime})$ of type
$\eta$ induces an isomorphism $Q_\eta(F)\colon (Q_\eta(S,\gamma),
\Lambda_\eta(S,\gamma))\rightarrow
(Q_\eta(S^{\prime},\gamma^{\prime}),\Lambda_\eta(S^{\prime},\gamma^{\prime}))$
of complex Lipschitz structures relative to $\eta$, which is defined
as follows (recall that $F_p\in
\Hom_{\mClRep^\times_w}(\gamma_p,\gamma^{\prime}_p)$ and
$(F_p)_0=\id_{T_p^\ast M}$):
\begin{equation*}
Q_\eta(F)(q)\eqdef (\id_{T_p^\ast M}, F_p)\circ q\, ,\quad \forall q\in Q_\eta(S,\gamma)_p=\Hom_{\mClRep^\times_w}(\eta,\gamma_p)\, .
\end{equation*}
\item Let $(Q,\Lambda)$ denote a complex Lipschitz structure of type
  $\eta$ on $(M,g)$. Then, the vector bundle $S\eqdef
  S_\eta(Q,\Lambda)\eqdef Q\times_{\rho_\eta} \Sigma$ associated to
  $Q$ through the complex tautological representation $\rho_\eta\colon
  \L\rightarrow \Aut_\C(\Sigma)$ of $\L$ becomes a complex pinor
  bundle of type $\eta$ when equipped with the Clifford structure
  morphism $\gamma\eqdef \gamma(Q,\Lambda)\colon \Cl(M,
  g)\rightarrow \End_{\C}(S)$ defined as follows:
\begin{equation}
\label{eq:1gamma1def0}
\gamma_p(y)([q,s]) \eqdef [q,\eta(\Cl(\Lambda_p(q)^{-1})(y))(s)]\, , \quad \forall\,\, y\in \Cl(T_p^\ast M,g_p^\ast)\, ,
\end{equation}
for all $q\in Q_p$ and $s\in \Sigma$. We call the pair
$S_\eta(Q,\Lambda)=(S,\gamma)$ thus constructed \emph{the complex
pinor bundle defined by the Lipschitz structure $(Q,\Lambda)$}. An
isomorphism of complex Lipschitz structures $F\colon (Q,\Lambda)\rightarrow
(Q^{\prime},\Lambda^{\prime})$ relative to $\eta$ induces a based
isomorphism of complex pinor bundles $S_\eta(F)=(S_\eta(Q,\Lambda),
\gamma_\eta(Q,\Lambda))\rightarrow
(S_\eta(Q^{\prime},\Lambda^{\prime}),\gamma_\eta(Q^{\prime},\Lambda^{\prime}))$
defined as follows:
\begin{equation}
\label{Sfdef}
S_\eta(F)_p([q,s])=[F_p(q),s]\, , \quad \forall q\in Q_p\, , \,\,\forall s\in \Sigma~~.
\end{equation}
\end{enumerate}
\end{prop}

\begin{proof}
The fact that \eqref{eq:deftau0} is $\Ad_{\eta}$-equivariant follows
from the relation $(q\circ \varphi)_0= q_0 \varphi_0= q_0\circ
\Ad_{\eta}(\varphi)$ for all $\varphi\in \L$, which implies that the
following equation holds:
\begin{equation}
\label{eq:Ad0relation}
\Lambda_{p}(q \varphi) = \Lambda_{p}(q)\circ\varphi_{0} = \Lambda_{p}(q)\circ\Ad_{\eta}(\varphi)\, .
\end{equation} 
This in turn implies that $(Q,\Lambda)$ is an elementary complex Lipschitz structure of
type $\eta$.

It remains to show that \eqref{eq:1gamma1def0} is well-defined. In
order to do this, notice that $\Ad(\varphi)\circ \eta=\eta\circ
\Cl(\varphi_{0})$ for any $\varphi\in \L$, which implies (using
$[\Cl(\Lambda_p(q))]^{-1}=\Cl(\Lambda_p(q)^{-1})$ and
$\Ad(\varphi^{-1})=\Ad(\varphi)^{-1}$):
\begin{equation*}
\Ad(\varphi^{-1})\circ \eta\circ \Cl(\Lambda_p(q)^{-1})=\eta\circ \Cl(\varphi^{-1}_{0}\circ \Lambda_p(q)^{-1})=\eta\circ [\Cl(\Lambda_p(q)\circ \varphi_{0})]^{-1}\, .
\end{equation*}
Using relation \eqref{eq:Ad0relation}, this gives: 
\begin{equation}
\label{rel1}
\Ad(\varphi^{-1})\circ \eta\circ \Cl(\Lambda_p(q)^{-1})=\eta\circ \Cl(\Lambda_p(q \varphi)^{-1})\, , \,\, \forall \varphi\in \L\, .
\end{equation}
Thus:
\begin{eqnarray*}
& & [q\varphi^{-1},\eta(\Cl(\Lambda_p(q\varphi^{-1})^{-1})(x))(\varphi s)] = [q, \varphi^{-1} \eta(\Cl(\Lambda_p(q\varphi^{-1})^{-1})(x))(\varphi s)]=\nonumber\\ 
& & [q,(\Ad(\varphi^{-1})\circ \eta\circ \Cl(\Lambda_p(q\varphi^{-1})^{-1}))(x)(s)] =[q,\eta(\Cl(\Lambda_p(q)^{-1})(x))(s)]~~,
\end{eqnarray*}
where in the last equality we used \eqref{rel1}. This shows that
\eqref{eq:1gamma1def0} is well-defined. The fact that $\gamma_p$
defined in \eqref{eq:1gamma1def0} is a Clifford representation is clear,
as is the fact that $P_{\eta}$ and $S_{\eta}$ define functors.
\end{proof}

\begin{remark}
The same construction used above for $S_\eta$ allows us to define a
functor $S^0_\eta:\L_\eta^0(M,g)\rightarrow \ClB_\eta(M,g)$ which
associates a complex pinor bundle to any reduced complex Lipschitz
structure.
\end{remark}

\noindent
The following theorem establishes an equivalence between complex pinor bundles and complex
complex Lipschitz structures defined over $(M,g)$. 

\begin{thm}
\label{thm:BundleLipschitz}
Let $\eta\in \Ob(\mClRep_w)$. The functors $Q_\eta$ and $S_\eta$ are
mutually quasi-inverse equivalences between the groupoids
$\mClB_{\eta}(M,g)^\times$ and $\L_\eta(M,g)^\times$. 
\end{thm}

\begin{proof}
The proof is analogous to that of \cite[Theorem 6.2]{Lipschitz}, so we
leave it to the reader.
\end{proof}

\begin{cor}
Let $\eta\in \Ob(\mClRep_w)$ and $\L_\eta^0$ be a reduction of
$\L_\eta$. Then there exists an equivalence of groupoids between
$\mClB_\eta(M,g)^\times$ and $\L_\eta^0(M,g)^\times$.
\end{cor}

\begin{proof}
Follows immediately from Theorems \ref{thm:Lredequiv} and \ref{thm:BundleLipschitz}. 
\end{proof}

\begin{prop} 
Let $(Q,\Lambda)$ be a complex Lipschitz structure of type $\eta$ and
let $(Q_0,\Lambda_0,I)$ be an $\L_\eta^0$-reduction of
$(Q,\Lambda)$. Then the complex pinor bundles $S$ and $S_0$ associated
to $(Q,\Lambda)$ and $(Q_0,\Lambda_0)$ are naturally isomorphic.
\end{prop}

\begin{proof}
Let $S = Q\times_{\rho} \Sigma$ and $S_0 = Q_0\times_{\rho_{0}}
\Sigma$ denote the vector bundles associated to $Q$ and $Q_0$ through
the tautological representations $\rho$ and $\rho_0=\rho|_{\L_\eta^0}$
of $\L_\eta$ and $\L^0_\eta$. Let $\gamma$ and $\gamma_0$ denote the
structure morphisms making $(S,\gamma)$ and $(S_0,\gamma_0)$ into
complex pinor bundles. A natural isomorphism $F\colon S_0 \to S$ is given by:
\begin{equation*}
F([q_0, s]_{\rho_0}) \eqdef [I(q_0), s]_{\rho}~~\forall q_0\in Q_0~~\forall s\in \Sigma~~.
\end{equation*} 
It is clear that $F$ is well-defined. To show that it is an
isomorphism of vector bundles, let $p$ be any point in $M$.  Since the
fiber $Q_p$ of $Q$ at $p$ is an $\L_\eta$-torsor, it follows that for
every $q\in Q_p$ there exists $q'_0\in Q^0_p$ such that $q=I(q'_0)u$
for some $u\in \L_\eta$. Writing $u=u_0 v$ with $u_0\eqdef r(u)\in
\L_\eta^0$ and $v\eqdef u_0^{-1} u\in \L_\eta$, we have $q=I(q'_0)u_0
v=I(q) v$, where $q_0\eqdef q'_0 u_0\in Q_p^0$. Thus:
\be
[q,s]_\rho=[I(q_0)v, s]_\rho=[I(q_0), \rho(v)s]_\rho=F_p([q_0,\rho(v)s]_{\rho_0})~~.
\ee
This shows that the the linear map $F_p:S^0_p\rightarrow S_p$ is
surjective and hence bijective (since $\dim_\C S_{0 p}=\rk_\C S_0=\rk_\C
S=\dim_\C S_p$). Thus $F$ is an isomorphism of vector bundles.  For
any $x\in \Cl(T_p^\ast M, g_p^\ast)$, any $q_0\in Q_p^0$ and any $s\in
\Sigma$, we have:
\begin{eqnarray*}
& (F_p\circ \gamma_{0,p}(x)) ([q_0,s]_{\rho_0}) = [I(q_0),\eta(\Cl(\Lambda_0(q_0)^{-1})(x))(s)]_\rho = [I(q_0),\eta(\Cl(\Lambda(I(q_0))^{-1})(x))(s)
]_\rho \\ & =\gamma_p(x)([I(q_0),s]_\rho) = (\gamma_0(x)\circ F)([q,s]_\rho)~~,
\end{eqnarray*}
where we used the relation $\Lambda\circ I=\Lambda_0$. This shows that
$(S,\gamma)$ and $(S_0,\gamma_0)$ are isomorphic as complex pinor
bundles.
\end{proof}

\subsection{Topological obstructions for elementary complex Lipschitz structures}

The topological obstructions to existence of complex Lipschitz
structures associated to faithful complex Clifford representations
were determined in \cite{FriedrichTrautman}. Since for even $d$ such
representations are $\C$-irreducible, the results of op. cit. together
with those of \cite{Karoubi} give:

\begin{thm}\cite{Karoubi,FriedrichTrautman}
Suppose that $d$ is even. Then $(M,g)$ admits an elementary complex
Lipschitz structure {\em if and only if} it admits a $\Pin^c(V,h)$
structure, i.e. if and only if there exists a principal $\U(1)$-bundle
$E$ over $M$ such that:
\begin{equation}
\label{evobs}
\w_2^-(M)+\w_2^+(M)+\w_1^-(M)^2+\w_1^-(M)\w_1^+(M)=\w_2(E)\, .
\end{equation}
\end{thm}

\noindent 
When $d$ is odd, the results of \cite{FriedrichTrautman} concern {\em
  non-elementary} complex Lipschitz structures. The following
Corollary of Theorem \ref{thm:BundleLipschitz} shows that on an
odd-dimensional pseudo-Riemannian manifold any bundle of irreducible
complex Clifford modules over $\Cl(M,g)$ can be understood as a
complex spinor bundle associated to a $\Spin^{c}(V,h)$
structure on $(M,g)$ through the tautological irreducible
representation of $\Spin^{c}(V,h)$.

\begin{cor}
\label{cor:SpincLipschitz}
Suppose that $d$ is odd. Then $(M,g)$ admits an elementary complex
pinor bundle $(S,\gamma)$ {\em if and only if} it admits a
$\Spin^{c}(V,h)$-structure $(Q,\Lambda)$ such that $S$ is a vector
bundle associated to $Q$ through the tautological representation of
$\Spin^{c}(V,h)$. Furthermore, the groupoid of elementary complex
pinor bundles is isomorphic to the groupoid of $\Spin^{c}(V,h)$
structures. When $(M,g)$ admits a $\Spin^c(V,h)$ structure, then
isomorphism classes of elementary complex pinor bundles form a torsor
over $H^2(M,\mathbb{Z})$.
\end{cor}

\begin{proof}
Since an irreducible complex representation is weakly faithful, the
first statement in Theorem \ref{thm:BundleLipschitz} implies that
$(M,g)$ admits an elementary complex pinor bundle $(S,\gamma)$ of type
$\eta$ if and only if it admits a Lipschitz structure $(Q,\Lambda)$ of
irreducible type $\eta$. Proposition \ref{prop:Gammaodd} implies that
the associated Lipschitz group $\L_\eta$ is isomorphic to
$\mathbb{R}_{>0}\Spin(V,h)$. By Proposition
\ref{prop:Cliffordgrouphomotopy} we deduce that $\L_\eta$ is
homotopy-equivalent to $\L^0_{\eta} = \Spin^{c}(V,h)$ via the
normalization morphism $r$ of equation \eqref{r}, a homotopy inverse
of which is provided by the inclusion $\iota\colon
\Spin^{c}(V,h)\hookrightarrow \L_\eta$. Hence Theorem
\ref{thm:BundleLipschitz} applies and we conclude that the groupoid of
elementary complex pinor bundles is isomorphic to the groupoid of
$\Spin^c(V,h)$ structures on $(M,g)$. In particular, every elementary
complex pinor bundle is associated to a $\Spin^c(V,h)$ structure on
$(M,g)$ by means of the tautological representation. The remaining
statement follows since $\Spin^c(V,h)$ structures on $(M,g)$ form a
torsor over $H^2(M,\mathbb{Z})$.
\end{proof}

\begin{cor}
A pseudo-Riemannian manifold $(M,g)$ admits an elementary complex
pinor bundle {\em if and only if} each of the following conditions are
satisfied:
\begin{enumerate}[1.]
\itemsep 0.0em
\item $\w_1(M)=0$ 
\item There exists a principal $\U(1)$-bundle $E$ over $M$ such that:
\begin{equation*}
\label{oddobs}
\w_2^-(M)+\w_2^+(M)=\w_2(E)~~.
\end{equation*}
\end{enumerate}
In particular, $M$ must be orientable. 
\end{cor}

\begin{proof}
Follows immediately from Corollary \ref{cor:SpincLipschitz} and the
results of \cite{Karoubi}.
\end{proof}

\subsection{Realification of complex pinor bundles}

The realification functor $R:\Rep_\C^w(\Cl(V,h))\rightarrow
\Rep_\R^w(\Cl(V,h))$ of Subsection \ref{subsec:realif} induces a {\em
  pinor bundle realification functor}
$\mathfrak{R}:\mClB_{\eta}(M,g)\rightarrow \ClB_{\eta_\R}(M,g)$, where
$\ClB_{\eta_\R}(M,g)$ denotes the category of real pinor bundles of
type $\eta_\R$ and based morphisms of such (see
\cite{Lipschitz}). When $p-q\equiv_8 3,4,6,7$, this restricts to a
functor which maps elementary complex pinor bundles to elementary real
pinor bundles.  Since this restricted functor is determined on fibers,
the results of Subsection \ref{subsec:realif} imply that it is
faithful but not full.  The restricted functor need not be essentially
surjective since a real elementary pinor bundle need not admit a
globally-defined complex structure which is a section of its Schur
sub-bundle. However, we have:

\begin{prop}
Assume that $p-q\equiv_8 3,7$ and that $M$ is orientable. Then
$\mathfrak{R}$ restricts to a strictly surjective functor from the
groupoid of elementary complex pinor bundles defined on $(M,g)$ to the
groupoid of elementary real pinor bundles defined on $(M,g)$.
\end{prop}

\begin{proof}
Let $\nu\in \Omega(M)$ be the volume form of $(M,g)$ determined by
some fixed orientation of $M$. Let $(\Sigma,\rho)$ be an elementary
real pinor bundle on $(M,g)$. Since $p-q\equiv_8 3,7$, the
endomorphism $J\eqdef \gamma(\nu)\in \Gamma(M,End(S))$ is a
globally-defined complex structure on $S$ which is a section of the
Schur bundle of $\Sigma$ (see \cite{fierz}). Thus $J_p$ lies in the
commutant of the image of the real Clifford representation
$\rho_p:\Cl(T_p^\ast M, g_p^\ast)\rightarrow \End_\R(\Sigma_p)$ for
any $p\in M$. This implies that $(\Sigma,\rho)$ is the realification
of the elementary complex pinor bundle obtained by endowing $\Sigma$
with the complex structure $J$.
\end{proof}

\begin{ack}
The work of C. I. L. was supported by grant IBS-R003-S1. The work of C. S. S. is supported by a Humboldt Research Fellowship from the Alexander Von Humboldt Foundation. The authors thank A. Moroianu for stimulating discussions and for interest in their work.
\end{ack}

\end{document}